\documentclass[a4paper,twoside,english]{amsart}
\usepackage[T1]{fontenc}
\usepackage[latin9]{inputenc}
\usepackage{babel}
\usepackage{verbatim}
\usepackage{amsthm}
\usepackage{amsmath}
\usepackage{amssymb}
\usepackage[unicode=true,
 bookmarks=true,bookmarksnumbered=true,bookmarksopen=false,
 breaklinks=false,pdfborder={0 0 1},backref=false,colorlinks=false]
 {hyperref}
\hypersetup{
 pdfauthor={V. Mantova and U. Zannier}}

\makeatletter

\pdfpageheight\paperheight 
\pdfpagewidth\paperwidth

\theoremstyle{plain}
\newtheorem{thm}{\protect\theoremname}[section]
 \newcommand\thmsname{\protect\theoremname}
 \newcommand\nm@thmtype{theorem}
 \theoremstyle{plain}

  \theoremstyle{remark}
  \newtheorem{rem}[thm]{\protect\remarkname}
  \theoremstyle{definition}
  \newtheorem*{example*}{\protect\examplename}
  \theoremstyle{definition}
  \newtheorem{example}[thm]{\protect\examplename}
  \theoremstyle{plain}
  \newtheorem{lem}[thm]{\protect\lemmaname}
  \theoremstyle{plain}
  \newtheorem{prop}[thm]{\protect\propositionname}
  \theoremstyle{plain}
  \newtheorem{cor}[thm]{\protect\corollaryname}

\makeatletter
  \theoremstyle{definition}
  \newtheorem{my@rem}[thm]{Remark}
  \renewenvironment{rem}{\begin{my@rem}}{\end{my@rem}}
\makeatother

\makeatother

  \providecommand{\examplename}{Example}
  \providecommand{\lemmaname}{Lemma}
  \providecommand{\propositionname}{Proposition}
  \providecommand{\remarkname}{Remark}
  \providecommand{\theoremname}{Theorem}
\providecommand{\theoremname}{Theorem}
 \providecommand{\corollaryname}{Corollary}
\usepackage[left=3.5cm,right=3.9cm,top=3cm,bottom=3cm]{geometry}

\usepackage{mathtools}

\def\N{{\Bbb N}}
\def\Q{{\Bbb Q}}
\def\R{{\Bbb R}}

\def\Z{{\Bbb Z}}
\def\G{{\Bbb G}}
\def\A{{\Bbb A}}

\def\O{{\mathcal O}}

\def\div{{\rm div}}
\def\ord{{\rm ord}}
\def\P{{\Bbb P}}
\def\C{{\Bbb C}}

\def\cR{{\mathcal R}} 
\def\cS{{\mathcal S}} 
\def\cA{{\mathcal A}}

\def\cG{{\mathcal G}}
\def\cM{{\mathcal M}}

\def\cD{{\mathcal D}}
\def\cZ{{\mathcal Z}}

\def\CVD{{\hfill\hfil{\lower 2pt\hbox{\vrule\vbox to 7pt
{\hrule width  5pt\varphifill\hrule}\varphirule}}}\par}

\title[Hyperelliptic Continued Fractions]{Hyperelliptic Continued Fractions and %a Skolem-Mahler-Lech theorem for 
Generalized Jacobians}
\author{Umberto  Zannier}\date{}
\begin{document}
\maketitle

%%% Titoli
%%%  Hyperelliptic continued fractions and generalised Jacobians
%%%  continued fractions in Hyperelliptic  fields and generalised Jacobians
%%%  Hyperelliptic continued fractions and a Skolem-Mahler-Lech theorem generalised Jacobians

%%% spedire Berry Bugeaud Frey  Bertrand 
%Bombieri Masser Tran Bertrand Schinzel Schmidt McMullen  

\noindent {\bf Abstract}. For a complex polynomial $D(t)$ of even degree, one may define  the continued fraction of $\sqrt{D(t)}$.   This was found relevant  already by Abel  in 1826, and later by Chebyshev, concerning integration of (hyperelliptic) differentials; they realized that, contrary to the classical case of square roots of positive integers treated by Lagrange and Galois, we do not always have pre-periodicity of the partial quotients. 

In this paper we shall prove that, however, a correct  analogue of Lagrange's theorem still exists in full generality:   pre-periodicity of the {\it degrees} of the partial quotients always holds.  Apparently, this fact  was never  noted  before.

This also yields a corresponding formula for the degrees of the convergents, for which we shall   prove  new bounds which are generally best possible (halving the known ones). %; we stress that these degrees are meaningful in that they determine the strength of the   approximations coming from the continued fraction. 

We shall further study other aspects of the continued fraction, like the growth of the heights of partial quotients. Throughout,  some   striking phenomena  appear, related to the geometry of (generalized) Hyperelliptic Jacobians.  Another conclusion central in this paper concerns  the poles of the convergents:  there can be  only finitely many rational ones which  occur infinitely many times. (This is crucial for  applications to a function field version of a question of McMullen.)  

Our methods rely, among other things, on linking  Pad\'e approximants and convergents with divisor relations in generalized Jacobians; this shall   allow an application of a version for algebraic groups, proved in this paper,  of the Skolem-Mahler-Lech theorem. % and  applied to suitable generalized Jacobians of hyperelliptic curves. %The questions are also related to points of bounded degree on hyperelliptic curves   and here a variation on an idea of Frey shall yield further consequences. 

\bigskip

\section{Introduction}

This paper is mainly  concerned with the continued fraction expansion of the square root of a complex polynomial $D(t)$, studied already by Abel \cite{[A]}  in  1826 and again by Chebyshev  \cite{Che} in 1852. For completeness we start by recalling very briefly  some basic facts about continued fractions.

\subsection{\tt Continued fractions of numbers and functions} For a real irrational number $\lambda\in\R\setminus\Q$,    its continued fraction  is obtained  by taking the integral part $a_0=\lfloor\lambda\rfloor$,   %putting $\lambda_1:=(\lambda-a_0)^{-1} >1$  
writing  $\lambda=a_0+(1/\lambda_1)$ (so $\lambda_1>1$)    and continuing with $\lambda_1$ in place of $\lambda$, and so on. This yields  an expansion $\lambda=a_0+1/a_1+1/a_2+1/\ldots$, denoted also $[a_0,a_1,\ldots ]$, which has various important properties.\footnote{When $\lambda=a/b$ is rational the procedure eventually  terminates and corresponds to the Euclidean algorithm for $a,b$.}   The $a_i$, called {\it partial quotients},  are integers, positive for $i>0$. The rational numbers $p_n/q_n=[a_0,a_1,\ldots ,a_{n-1}]$ obtained by truncating the expansion before $a_n$  (we agree that $(p_0,q_0)=(1,0)$) are called the  {\it convergents}, and they may be shown to provide the `best' rational approximations to $\lambda$.  (See \cite{[C]}.)

 For an irrational Laurent series $\lambda(t)\in \C((t^{-1}))\setminus \C(t)$, we may obtain a continued fraction in a completely  similar way, on replacing the integral part by the  {\it polynomial part},   defined as the unique polynomial $a_0(t)$ such that $\lambda(t)-a_0(t)$ is a power series in $t^{-1}$.  The partial quotients $a_i(t)$ now are polynomials, of degree $>0$ for $i>0$, and the convergents $p_n(t)/q_n(t)$ have similar best-approximation properties with respect to the valuation of $\C((t^{-1}))$. For instance,  $p_n(t)-q_n(t)\lambda(t)$ vanishes at  $t=\infty$  to an order, which is $\deg q_{n+1}(t)$,  maximal with respect to all $p(t)-q(t)\lambda(t)$, for $0\le \deg q<\deg q_{n+1}$.  
 
 We refer to \cite{[vdP]} and \cite{[vdPT]} for these  and other properties and  for references. We shall  also refer to the pairs $(p_n(t),q_n(t))$  as {\it convergents}, when there is no risk of confusion; they are also called {\it continuants} of the continued fraction.  We further recall  that  they provide the so-called   {\it Pad\'e approximants} to $\lambda(t)$ and are relevant in various contexts.\footnote{The fraction $p_n/q_n$ determines the polynomials $p_n,q_n$ only up to a factor; usually here we  implicitly mean that $p_n,q_n$ are calculated formally from the $a_n$ in the well-known natural way.}

\medskip

Now, the simplest real irrational numbers are the quadratic ones, and it is  classical that   the continued fraction  for any such number is eventually periodic, a result due  to Lagrange, with further precision by Galois. For  the numbers $\sqrt D$,  for a positive integer $D$, not a perfect square,  such  periodicity property is strictly related to  the solvability,  in the integer unknowns $x,y$, of the `Pell equation'  (proposed in fact by Fermat) 
\begin{equation*}\label{E.pell}
x^2-Dy^2=1,\qquad y\neq 0,
\end{equation*}
which indeed  admits infinitely many integer solutions for any  given  non-square $D\in\N$. The equation is well known to be fundamental in the theory of integral quadratic forms.

\medskip 

In analogy, let  now $D(t)$ be a non-square complex polynomial of even degree, denoted $2d$. %, and assume that it is not a square.  
We may then expand its square root $\sqrt{D(t)}$ as an irrational Laurent series in $t^{-1}$, and consequently obtain a continued fraction, as above.  One may then ask which of  the above mentioned facts persist in this case.

\subsection{\tt Abel and Chebyshev} It was Abel who, apparently for the first time,  studied  in depth  such polynomial case, in 1826 \cite{[A]}; then the topic was again took  by Chebyshev \cite{Che}. 

To describe this, it shall be convenient to  call {\it Pellian} a polynomial $D=D(t)\in\C[t]$ as above, for which the Pell equation % \eqref{E.pell}  
 is   solvable in nonzero polynomials $x(t),y(t)\in \C[t]$.\footnote{This notion heavily depends on the ground field, but here we tacitly stick to $\C$.} 

Abel was mainly motivated by the problem of expressing (hyperelliptic) integrals in `finite terms', and found that certain differentials on the curve $u^2=D(t)$ could be likewise integrated when   $D(t)$ is Pellian;  since that time it has been indeed understood   that the topic is intimately related with abelian integrals and Jacobians (of the curves in question).  We shall   see explicit links  later.\footnote{Already in the numerical case, Dirichet class-number formulae and other results indicate a strict connection of the topic with the suitable Picard groups.}  (See also \cite{A-R},  \cite{Be}, \cite{[vdPT]},   \cite{Schm}, \cite{[Z3]}.)  

 Abel, although without proof, realized that, in marked contrast with the case of integers,   not all complex polynomials are Pellian (even among the non-square  ones of even degree).\footnote{The case of polynomials over a finite field is, on the contrary, completely similar to the integer case.}    He and Chebyshev also understood that, this time  as in the case of integers, there is a strict relation  with the continued fraction; indeed,  in essence their contributions contained in particular %a proof of  
  the following

\medskip

\noindent{\bf Abel-Chebyshev theorem.}   {\it The  complex polynomial  $D(t)$  (non-square of even degree) is Pellian  if and only if the continued fraction for $\sqrt{D(t)}$ is eventually periodic}. 

\medskip

These Pell equations and continued fractions have been studied since then in several papers. Beyond the above mentioned ones,  we  quote also Schinzel's \cite{Schi}, concerning   relations between  the continued fractions for $\sqrt{D(t)}$ and its values $\sqrt{D(n)}$ ($n\in\N$). 

\subsection{\tt Results of this paper} As a matter of fact, from many viewpoints  `pellianity'  is extremely rare for any given $d>1$: for instance,  it may be shown   that inside the $(2d-2)$-dimensional family of polynomials $D(t)$ of degree $2d$ suitably normalized, the Pellian ones  form a denumerable union of algebraic families of dimension $\le d-1$.\footnote{A formal proof of this is the object of work in progress, but some detail appears already in  \cite{[Z3]}, especially  \S 2.2. It shall anyway clearly appear later that pellianity  is indeed uncommon.} See also e.g. the joint paper with D. Masser \cite{MZ} for a proof that on `most' 1-dimensional families of polynomials of degree $2d\ge 6$  there are   only finitely many Pellian ones. (These facts fall into the realm of `Unlikely Intersections' and `relative Manin-Mumford', as in \cite{[Z]}; they are also related to Manin's theorem of the kernel, as in forthcoming papers  with Y. Andr\'e, P. Corvaja and Masser.) 

So, from these considerations and the Abel-Chebyshev theorem we deduce that in a sense periodicity of the continued fraction for $\sqrt{D(t)}$ is a very `rare' phenomenon as well.

Now, we have realized, not without surprise, that, however, some periodicity survives %can be restored
 in full generality; indeed, we have the following

 \begin{thm}\label{T.dega}  %For $D(t)$ a complex non-square polynomial of even degree, if $a_0(t),a_1(t),\ldots$ denote the partial quotients of the continued fraction for $\sqrt{D(t)}$,  
 The sequence %$\deg a_0,\deg a_1,\ldots$ 
 of  degrees of the partial quotients for $\sqrt{D(t)}$  is eventually periodic.
\end{thm}

This analogue of Lagrange's theorem seems to have never been noted or suspected 
before,  %at any rate explicitly, 
in spite of the fact  that  the most common case is by far when all degrees are eventually $1$ (or eventually constant), as  shall appear from considerations   below (see e.g. \S \ref{SSS.pause},  Example \ref{EX.g=0} and \S \ref{SSS.rem.ex}).  Indeed, for $d\le 3$ (or when $u^2=D(t)$ has genus $0$)  it may be seen that $\deg a_n$ is eventually constant  in the non-Pellian cases; however for $d\ge 4$ dimensional considerations suggest that this is not generally  the case and in fact explicit examples have been found in this sense.\footnote{For reasons of space, we omit  a discussion  of this here, which is somewhat laborious, depending on Jacobians of dimension $\ge 3$ containing a translate of an elliptic curve inside the set of sums of two points of the curve. To give a specific example, the polynomial $D(t)=t^8 -  t^7 - (3/4)t^6 + (7/2)t^5 - (21/4)t^4 + (7/2)t^3 - (3/4)t^2 -t+ 1
%X^8 - X^7 - 3/4*X^6 + 7/2*X^5 - 21/4*X^4 + 7/2*X^3 - 3/4*X^2 - X + 1     t^8 + (146/3)t^7 + (6238/9)t^6 + 2531t^5 + (156769/36)t^4 + (12491/3)t^3 +(20686/9)t^2 + (2078/3)t + 89
$ yields infinitely many partial quotients of degrees $1$ and $2$, with the periodic pattern of degrees $4, 1, 1, 2, 1, 1, 1, 1, 1, 1, 1, 1, 2, 1, 1, 1, 1, 1, 1, 1, 1, 2, 1, \ldots$. See O. Merkert's thesis \cite{Me} for more. We plan to publish  a detailed  presentation  in the future.} %Our proof shall also yield further precision  on which we shall comment later.

\smallskip

We stress that the quantities $\deg a_n(t)$ are relevant ones, e.g.  for the approximations to $\sqrt{D(t)}$. Indeed, using  the asymptotic symbols  % `$O(t^m)$' 
in the sense of the valuation of $\C((t^{-1}))$ (i.e. at $t=\infty$), we have 
 \begin{equation}\label{E.approx}
p_n(t)-q_n(t)\sqrt{D(t)}\sim c_n\cdot t^{-\deg q_n-\deg a_n}\quad c_n\neq 0.
\end{equation} 

We also recall at once that, somewhat conversely,  if $p(t)-q(t)\sqrt{D(t)}=O(t^{-\deg q-1})$ for polynomials $p,q\neq 0$, then $p/q$ is a convergent (see \cite{[vdPT]}).

 \begin{rem} \label{R.hankel}  (i)  {\tt Hankel determinants}.   The degrees of the $a_n$ are linked  to the so-called {\it Hankel  matrices} associated to the Laurent coefficients for $\sqrt{D(t)}$: a large degree amounts to the vanishing of several  determinants in these matrices. Our proofs show that these vanishings always  have periodic pattern and are related to the geometry of (generalized) Jacobians for the curves $u^2=D(t)$.  
 
 (ii) {\tt Roth's theorem for algebraic functions}. One may also wonder whether this periodic behavior holds  generally   for continued fraction expansions of algebraic functions\footnote{The methods of this paper should probably prove this for arbitrary elements of $\C(t,\sqrt{D(t)})$, though for simplicity we work only with the special emblematic case of $\sqrt{D(t)}$.}; such  issue  is related to a possible strong version of Roth's theorem over function fields, known only for algebraic functions of degree $\le 3$ over $\C(t)$ (see M. Ru's paper \cite{Min}).  
 
 \end{rem}

 %A case when  the degrees are particularly interesting is  if  $D=D_\lambda \in \Q(\lambda)[t]$ is   non-Pellian, and depends on a   parameter $\lambda$. In this case  the coefficients of the $a_n$  are functions of $\lambda$  and it may be of interest to specialize $\lambda\to c$ and to compare with  the continued fraction for $\sqrt{D_c(t)}$.  Of course this cannot be done when $c$ is a pole of some coefficient, and recent work of O. Merkert \cite{Me}  shows this happens when some degree increases by specialization. Subtle problems arise in understanding this behaviour, some of them related to integral points of affine subsets of abelian varieties over function fields (see also \cite{MZ}).   Merkert's work also seems to indicate  that the anti-period `tends' to be long.    \medskip

  Regarding again the degrees of the $a_n$,  it is well known (see e.g. \cite{[vdPT]})  that $1\le\deg a_n\le d$ for all $n$ and that  the upper bound is attained for some $n>0$ precisely when $D(t)$ is Pellian (in which case it is attained over a whole arithmetic progression of $n$). In the non-Pellian cases we shall improve on this,  by showing a best possible general upper bound:
 
 \begin{thm}\label{T.dega2} 
 We have $\deg a_n(t)\le {d\over 2}$ for all large $n$, unless  $D(t)=r(t)^2D^*(t)$  for polynomials $r,D^*$, with $D^*$ Pellian  of degree $>{3\over 2}d$.
 
 In particular, the bound holds for squarefree non-Pellian $D(t)$.
  \end{thm}
 
 \begin{small}
 
 \begin{rem} \label{R.deg<}  We   cannot avoid the exceptions in the statement: if $D^*$ is Pellian  an infinity %there are infinitely many 
 of convergents $(p,q)$ to $\sqrt{D^*}$  have partial quotient of degree $d^*=\deg D^*/2$; but then $rp/q$ is a convergent to $\sqrt D$ with partial quotient of degree $\ge d^*-\deg r=d^*-(d-d^*)=2d^*-d>d/2$. %, as required.
 
 Further,  %we have already remarked that 
 although `usually' we have $\deg a_n=1$ for all  large $n$,  %so that often the upper bound is not sharp; however, we also  point out that it  
 the above bound cannot be generally improved,  even in the squarefree  non-Pellian case. To justify this claim,   let $D(t)=t^{4b}+t^b+\lambda$,   where $\lambda\in\C$ is transcendental and $b$ is a positive integer. Let $u_n(t),v_n(t)$ be the convergents to $\sqrt{t^4+t+\lambda}$, so in particular 
 $
 u_n(t)-v_n(t)\sqrt{t^4+t+\lambda}=O(t^{-\deg v_n-1}),
 $
 which yields 
$
 u_n(t^b)-v_n(t^b))\sqrt{D(t)}=O(t^{-\deg v_n(t^b) -b}).
 $
 By the asymptotic  \eqref{E.approx} above and the subsequent remark, we deduce that $u_n(t^b)/v_n(t^b)$ are convergents to $\sqrt{D(t)}$ whose corresponding partial quotients  have degree $\ge b=(\deg D)/4=d/2$. 
On the other hand, $D(t)$ is squarefree and cannot be Pellian, as can be easily proved e.g. with the argument appearing in \cite{[Z]}, Remark. 3.4.2, p. 85.

%%%%%%%%%%%%%%%%%%%%%%%%%%%%%%%%%

%: for otherwise we could view $\lambda$ as a variable and  a solution of the Pell equation would have coefficients in $\Q(\lambda)$ (as can be seen by conjugation and using  the structure of solutions, recalled below). But then we would have an identity $x(\lambda,t)^2-y(\lambda,t)^2D(t)=c(\lambda)$ for suitable nonzero  $x,y\in\Q[\lambda,t]$  without common factors in $\lambda$  and  nonzero $c\in\Q[\lambda]$. Now, $c$ cannot be constant, because $D$ has odd degree in $\lambda$. Hence there would exist a complex root $\lambda_0$ of $c$,  yielding  that $t^{4b}+t^b+\lambda_0$ is the square of a polynomial, which is impossible. \footnote{This simple argument appears also in \cite{[Z]}, Remark. 3.4.2, p. 85. The $\lambda$ can also be specialized to an algebraic or even rational number: it suffices that this has large height, see \cite{[Z]}.}

%%%%%%%%%%%%%%%%%%%%%%%%%%%%%%%%%%%

 % Any non-Pellian $D(t)$ of degree $6$ (e.g.  $D(t)=t^6+t+\lambda$)  shows that the upper bound is sharp also  for $d=3$.  On the other hand, I  have presently no general examples of   non-Pellian $D(t)$   with odd $d>3$ and infinitely many partial quotients of degree $(d-1)/2$ (a lower bound $d/3$ can be obtained as above by taking $D(t)=t^{6b}+t^b+\lambda$).

 \end{rem} 
 \end{small}
 
 \noindent{\tt Heights}. When $D$ has algebraic coefficients, still another aspect concerns heights, which are relevant for many purposes.\footnote{Just to mention an instance, it will appear that continued fractions may be used to check computationally whether a point is torsion on a hyperelliptic Jacobian, and here heights affect the complexity.}
 
To fix  the basic definitions, we recall that for a nonzero polynomial $f(t)\in\overline \Q[t]$ one  considers the usual  projective absolute (logarithmic) height of the vector of its coefficient; this is denoted $h(f)$.  One can also consider the affine height of the same vector, denoted here $h_a(f)$.  We have $h_a(f)\ge h(f)\ge 0$.
 
 For the convergents $p_n,q_n$, a theorem of Bombieri-Cohen \cite{Bo-Co}, on which we shall comment below in more detail, predicts  the order of growth of the projective height. However this does not yield  the same information on the height of the partial quotients, especially concerning bounds from below.  We have the following result, where for the lower bound we stick to the affine height and, for simplicity,  to the squarefree case:
 
 \begin{thm} \label{T.h}  Suppose that $D(t)\in\overline\Q[t]$ is squarefree and non-Pellian. Then   $h(a_n)\ll n^2$. Also,  there exists an integer $M=M_D$ such that for all large $n$ we have
 $$
 \max_{s=0}^M h_a(a_{n-s})\gg  n^2.
  $$
% where the implicit constant does not depend on $n$. 
 \end{thm}

 %%%% 
 
 \begin{rem}\label{R.12}{\tt Peculiar (sub)sequences of $a_n$}. (i) The same kind of  lower bound of the theorem may be gotten restricting to the subsequence of $a_m$ when $m$ lies in a fixed arithmetical progression (we have stated the special case for simplicity). 
 
 (ii) Of course in the Pellian case the $a_n$ are periodic hence of bounded height. We have also found (with the help of numerical calculations by Merkert) some unexpected cases of non-Pellian $D(t)$ such that all the $a_n(t)$  with $n$ in certain arithmetical progressions are of the shape $c_n\cdot t$,  hence in particular   have bounded ($=0$) projective height.\footnote{See  \cite{Schm} for a notion of pseudo-periodicity, apparently similar to this, but in fact different..}  A relevant example has degree $12$ (and is defined over  a number  field of degree $5$)  \footnote{The sequence $(\deg a_n)$  in this case is $[6, 1, 1, 1, 1, 1, 1, 3, 1, 1, 1, 1, 1, 1, 1, 2, 2, 1, 1, 1, 1, 1, 1, 1, 3, 1, \ldots ]$.}; this  corresponds to a rather peculiar   Jacobian   of a curve of genus $5$, and we think  it would be not free of  interest  to explore in general  the nature of this kind of geometry. (For brevity we do not reproduce here the details of this example.)
 
 \end{rem} 
 
 \medskip
 
 In Example \ref{EX.ell} we shall sketch a proof that  in some cases (e.g.  $D(t)=t^4+t^2+t$) we have the striking fact that the {\it affine} height grows even faster:
 
 \medskip
 
\noindent {\bf Addendum}.  {\it For the partial quotients $a_n$ of $\sqrt{t^4+t^2+t}$, for any integer $k>0$, we have $h_a(a_n)+h_a(a_{n-k})\ge ckn^2$, for some absolute constant $c>0$ and all large enough $n$. }

\medskip

 This  implies a similar lower bound for $h_a(p_n),h_a(q_n)$; in particular, for the {\it affine} height this yields  (for this example) $\limsup h_a(a_n)/n^2=\infty$, contrary to the bound    $h(a_n)\ll n^2$  for the {\it projective} height. (Maybe $h_a(a_n)\gg n^3$ at least on a subsequence, but we have not much evidence for this.) 
 
 Several other comments are in order, but we postpone them  and further precision after the proof:  see Remark \ref{R.h} and Example \ref{EX.ell}. % Now we only add that of course in the Pellian case the $a_n$ are periodic hence of bounded height. 
 
  \medskip

\noindent{\tt Convergents and their poles}.  So far we have discussed partial quotients, and let us now turn to the convergents. In view of the well-known recurrences $q_{n+1}=a_nq_n+q_{n-1}$, so $\deg q_{n+1}=\deg a_n+\deg q_{n}$, Theorem  \ref{T.dega} also clearly  implies a formula 
 \begin{equation}\label{E.degq}
 \deg q_n=c\cdot n+r_n,
 \end{equation}
  for some rational $c>0$, with $r_n\in\Q$ eventually periodic (and similarly for the $p_n$, note    that  in fact $\deg p_n=\deg q_n+d$). Theorem \ref{T.dega2} also yields a lower bound for $c$.
  
  This is for what concerns degrees, but now we shall  be interested in  the poles of the convergents $p_n/q_n$, i.e. the zeros of the convergent denominators $q_n$, which of course can be considered analogues of their prime factors in the numerical case. We shall study heights and the occurrences  of a given zero.

 \medskip
 
 \begin{small} 
  Let us first briefly discuss  the Pellian case, when the continued fraction is periodic by Abel-Chebyshev  theorem. As is well known, the periodicity entails  that  % there is an integer 
  if $b$ is the period   we have
  \begin{equation}\label{E.pellconv}
  q_n(t)={\beta_r\mu^m-\beta_r'\mu^{-m}\over 2\sqrt{D(t)}},\qquad n=mb+r, \quad m\in\N,
  \end{equation}
  for suitable $\beta_r\in \C[t,\sqrt{D(t)}]$, where a dash denotes conjugation over $\C(t)$ and where $\mu=p(t)+q(t)\sqrt{D(t)}$ corresponds to the minimal solution $(p,q)$ of the Pell equation (so in particular we have $\mu'=\mu^{-1}$). Also,   we have $\beta_0=\beta'_0=1$. 
  
  This formula of course makes it relatively easy to extract  properties of the zeros, for instance concerning their location and also their arithmetic. 
   In fact, for a zero $\theta$ one has $\mu(\xi)^{2m}=\beta_r'(\xi)/\beta_r(\xi)$, where $\xi$ is a point of the curve $u^2=D(t)$ above $t=\theta$. 
  
  If for instance we work over $\overline\Q$, this easily entails that the zeros have bounded (logarithmic Weil) height, as  also  suggested by the bound $h(q_n)=O(n)=O(\deg q_n)$ coming from \eqref{E.pellconv}. 
  
  Also,  \eqref{E.pellconv}  yields  that if a given $\theta$ is a zero of infinitely many among the $q_n$ %(it suffices of two of them corresponding to the same $r$)  
  then $\mu(\xi)$ is a root of unity. Actually, for $r=0$ we see that anyway $\mu(\xi)$ is a root of unity and that the zero is common to all $q_{bn}$, for $bn$ multiple of the order of the root of unity.\footnote{In fact, using known results on torsion points on curves, one can easily show that for $2r\not\equiv 0\pmod b$ a zero can appear only finitely many times.}  In particular, the zeros common to sufficiently many $q_n$ are linked to cyclotomic fields, there are infinitely many of them but  only finitely many ones  of bounded degree. 
  \end{small} 
  
    \medskip
  
  In the non-Pellian case we have no  simple  formula to help us, but still we may say something on these issues.

  Concerning the height of the zeros,  as mentioned above, a (special case of a) theorem by Bombieri-Cohen (see \cite{Bo-Co})  says that,  in marked  contrast with the Pellian case,   if the squarefree part of $D(t)$ is already non-Pellian  % \footnote{We shall comment later on the `semi-Pellian' case.} 
   the  height of the $q_n$   grows quadratically: $h(q_n)\gg n^2$. This is of course linked with Theorem \ref{T.h} above, and for our special context we shall reprove in a simple way this fact later (see Remark \ref{R.BoCo});  now we observe at once that, since $\deg q_n\ll n$, this yields by general properties  (see \cite{[BG]}, Ch. 1)  that the average zero has large height:
  \begin{equation*}\label{E.h}
  {1\over \deg q_n}\sum_{q_n(\theta)=0}\ord_\theta(q_n)\cdot h(\theta)\gg n,
  \end{equation*}
  so that in particular the boundedness of the height of the zeros now badly  fails.  This also makes it difficult to study the location of zeros \footnote{This is relevant e.g. in specialising functional approximations.}, for which   deep problems of Diophantine Approximation on abelian varieties arise, on which we shall comment later. 
  %% inserire commenti su approssimazioni in varieta' abeliane  ??
  
 % We have a further result, this time concerning the common zeros of  several  $q_n(t)$; note that one would not expect a zero to appear  many  times, except under special circumstances (see the discussion below for the Pellian case). In particular, the  zeros appearing infinitely  often  seem to us  of interest in their own sake, and are also relevant for certain applications, as we shall  point out.\footnote{For instance, if we want to derive numerical approximations by specializing at $t=\theta$, it is important that $q_n(\theta)\neq 0$; see Remark \ref{R.values}.}  Of course they can be considered analogues of the prime numbers dividing infinitely many convergents, in the numerical case.

\smallskip
  
  Concerning the appearance of zeros, %arithmetic, 
  we may   prove that some of the properties that we have observed for the Pellian case persist for the non-Pellian one; this    is  much more hidden  and is indispensable for certain applications, as mentioned below.
  
  %First, as  in forthcoming joint work with F. Malagoli (see also \cite{Ma}), it is not too difficult to show that algebraic numbers sufficiently ramified above a prime $\ell$ and non integral at $\ell$ cannot be zeros of any $q_n$. 
  
  We consider  the zeros appearing infinitely often (analogous to the primes dividing infinitely many $q_n$ in the numerical case).  %,   elements  in a ground number field (for instance $\Q$) cannot be likewise excluded, though control of these zeros would be needed in an application mentioned below.  
  By the methods developed in this paper  %we can fill this gap, 
  for instance we can show the following:
  
  \begin{thm} \label{T.zeros} Let $D\in \kappa[t]$, where  $\kappa$ is a number field. Then, for each $l$ there are only   finitely many  $\theta$ of degree $\le l$ over $\kappa$ 
  which are common  zeros of infinitely many   $q_n(t)$.
  \end{thm}
  
  Actually, it shall appear from the proofs that we can add further precision (for instance proving sometimes %%% sempre ?
  finiteness independently of $l$), on which we shall comment later (see Example \ref{EX.g=2} and Remark \ref{R.ber}(ii)). %, and again  cyclotomic fields shall turn out to be  relevant concerning these $\theta$. 
  Also, as remarked therein, dimensional considerations suggest that  even in the non-Pellian case there may exist  zeros which appear infinitely often (and this is related to the geometry of generalized hyperelliptic Jacobians).  %%% cf. appunti Clay ??
  
  \medskip
  
 A   relevant further motivation for studying these common zeros  is to relate  the continued fraction %(and hence the quality of good approximations) 
  for   $\sqrt{D(t)}$ with the one for $(t-\theta)\sqrt{D(t)}$ (so eventually relating with more general elements of $\C(t,\sqrt D)$):  it turns out that  the issue is substantially affected  by whether  or not $\theta$ is a zero of infinitely many $q_n$. 
 
 Using this link, Malagoli  \cite{Ma} has recently applied Theorem  \ref{T.zeros} to answer  in the affirmative an  analogue for the function field $\Q(t)$ of a question of MacMullen (see \cite{McM}, p. 22)  as to whether in every quadratic extension  there is an element whose partial quotients all have   degree  $\le 1$ (absolute value $\le 2$  in the numerical case), or at least degree bounded by an absolute constant. \footnote{As  in forthcoming joint work with F. Malagoli (see also \cite{Ma}), it is not too difficult to show that algebraic numbers sufficiently ramified above a prime $\ell$ and non integral at $\ell$ cannot be zeros of any $q_n$; however this fact alone   does not allow the said application. } 
 
 We further  remark that these applications   require  considering also the cases of non-square free $D(t)$, which complicates (also conceptually) the proofs.

 \medskip
 
A last result of this paper  concerns  the form  $x^2-Dy^2$ evaluated at convergent pairs $(p_n,q_n)$;  the corresponding values $R_n:=p_n^2-Dq_n^2$  in the numerical case  are the `smallest values at integral points'. %, and are relevant concerning e.g. the class-number.  
  In the present case, $R_n$ is  a polynomial  of degree $d-\deg a_n\le d-1$; actually, all nonzero values $p(t)^2-D(t)q(t)^2$ of degree $<d$ are proportional to some  $R_n$.  Also, $R_n$ can be   constant only when $D(t)$ is Pellian, in which case the sequence of the $R_n$ is periodic.   In the numerical case, the prime  factors of the numbers $R_n$ are linked to generators and relations  for the quadratic class-group. Here, in   partial analogy, we may then ask about the {\it factorization into irreducible  factors}  
  of these polynomials. Sticking again for simplicity to the squarefree case, we have the following result, proved using a deep theorem of Faltings:
 
 \begin{thm}\label{T.irr} Let $D(t)$ be squarefree, non-Pellian  and with coefficients in a number field $\kappa$.  %Then for large $n$ we have a $R_n(t)=I_n(t)\Phi_n(t)$, where  $I_n,\Phi_
 There exists a finite set  $\Phi=\Phi_\kappa$ of polynomials such that, for all large $n$,  $R_n(t)$ has exactly one irreducible factor (over $\kappa$) outside $\Phi$; this factor has degree $\ge d/2$ and may appear  only a number of  times bounded independently of $n$.
 \end{thm}
 
We shall add some further remarks after the proof of the theorem.

 \subsection{\tt Methods and organization of the paper}  The starting point of our proofs of the above theorems is by interpreting properties of convergents in terms of certain divisor equivalences. 
 
 This link is well known   in the case of  the Pell equation for squarefree $D(t)$, whose  solvability amounts to possible  torsion of    a suitable divisor class in the Jacobian of the underlying hyperelliptic curve; we shall recall this  in  Prop.  \ref{P.pell} below. 
    Our survey paper \cite{[Z3]} points out with some examples certain  generalizations of this   to Pell equations with non-squarefree $D(t)$, this time in terms of generalized Jacobians associated to the curve (as described e.g. in Serre's book \cite{SeAGCF}); see also  \cite{A-R}, \cite{Be}, \cite{Ber}, \cite{BMPZ}, \cite{McM2}  for further instances and links with other contexts.
   
    %The paper \cite{A-R} of Adams-Razar again observes the link of the continued fraction with divisor relations on an elliptic curves, limiting to the Pellian cases. 
    The paper \cite{Be} of Berry goes beyond the Pell equation and again relates the convergence to certain divisor relations (in part  following Chebyshev), however limiting to  small degree and with emphasis on  the computational viewpoint  (which is one possible applications of the present setting). %(and for small degree).  
    To our knowledge in the non-Pellian case these divisor relations have not been analyzed to  any further extent   explicitly  in the literature (and in particular generalized Jacobians seem not to appear anywhere).
    %, especially  for detecting zeros of the convergents and strength of the approximations     (even though such  criteria may appear not to be especially surprising).  
     
     \medskip

Here we shall associate to the convergents suitable  equations  in a generalized Jacobian corresponding to $D(t)$; then we shall  develop  related criteria  leading us to the study of the Zariski closure of the set of  multiples of a certain `canonical' point in the generalized Jacobian in question. 
   %This can be often dealt with by using  theorems of Faltings and their extensions to semiabelian varieties, which  would also work  with groups of arbitrary rank of course. However we   only need rank $1$ and moreover  these deep results  would not take care of the additive part. Instead, for most results we shall use a different  tool, directly related  to our purposes.\footnote{Faltings' theorems shall be used for the proof of Theorem \ref{T.irr}.}
   % which has the advantage that it not only applies in full generality for our purposes, but also leads to more explicit information, sometimes effective (in contrast to the   more advanced theorems alluded to above). 
 %% citare anche Dynamical Mordell-Lang ??
 
 We shall describe  this closure  by means of  %Such   result  describes the structure of the Zariski closure of an arbitrary   sequence of multiples of a given point in an algebraic group, and is a 
  a generalized form of the well-known Skolem-Mahler-Lech Theorem for zeros of recurrences, which   applies to  an arbitrary infinite  sequence of multiples of a given point in any algebraic group (in zero characteristic). % It works for every algebraic group in zero characteristic. 
 Recently some new  versions of the said  theorem appeared  in the literature, but we shall develop our one in \S \ref{S.SML} 
 below, with a self-contained very short  treatment (present  already in the first edition of the  writer's book \cite{[Z5]} independently of other versions). \footnote{One could also use   theorems of Faltings and their extensions.  However we   only need rank $1$ for most arguments  and moreover  these  results  would not take care of the additive part. Faltings' theorems shall be used for the proof of Theorem \ref{T.irr}.}
 
 \medskip

 In \S \ref{S.proof} %%% ??
 we shall deduce the proofs of the various assertions, and also include remarks, examples and some further precision.  
 
 \medskip
 
 We  add that this study has shown sometimes an unexpected behavior of the convergents, also through  numerical examples  %(carried out  by O. Merkert)   
 related to striking geometrical features of hyperelliptic Jacobians, which may deserve and hopefully raise independent analysis.  % Finally, \S \ref{S.rem} %%% ??   shall be devoted to remarks of various nature.
 
 \medskip
 
  \noindent{\bf Acknowledgements}.  It is a pleasure to thank Daniel Bertrand for clarifications concerning % splitting of 
  generalized Jacobians.  I am grateful to  Olaf Merkert  for several explicit computations  %related to continued fractions 
  and to Francesca  Malagoli for comments. I also thank the ERC Advanced Grant  267273 `Diophantine Problems' for support during the preparation of the paper.   %involving Pell equations and F. Malagoli for ...

 \section{Convergents and divisor relations in generalized Jacobians} \label{S.jac}

 \subsection{\tt Notation and preliminary remarks}\label{SS.notation}  We start by introducing the relevant notation and recalling some basic facts for the reader's convenience. 
 
 As above, $D(t)\in \kappa[t]$ shall denote a   polynomial over a subfield $\kappa$ of $\C$, of even degree $2d$ and not a square in $\C[t]$. An affine transformation $t\mapsto at+b$ does not modify any of the results %and quantities 
 we are interested in, so we shall often assume  that $D$ is monic and with second vanishing coefficient.

 We allow that $D(t)$ has square factors and we put $D(t)=D_1(t)^2\widetilde D(t)$, with monic $D_1,\widetilde D\in \kappa[t]$, $\widetilde D$ without multiple factors. (We shall often omit the tilde when $D$ is squarefree, i.e. when $D=\widetilde D$.)  We put  $\deg \widetilde D=2\tilde d>0$, $\deg D_1=d_1$. 
 
 We let $\widetilde H$ be a complete smooth curve with function field $\kappa(t,u)$, where 
 \begin{equation}\label{E.H}
 u^2=\widetilde D(t).
 \end{equation} 
 The function field $\kappa(t,u)$ is  a quadratic extension of $\kappa(t)$, and we shall denote the nontrivial involution $t\mapsto t$,  $u\mapsto -u$ with a dash.
 
 %We can use $\widetilde H_a$ to denote the affine piece of $\widetilde H$  corresponding to the equation. 
 We note that the genus $\tilde g$ of $\widetilde H$ is given by 
 $
 \tilde g:=\tilde d-1.
 $
 Usually we shall be interested in the case $\tilde g\ge 1$, though it is easy to make sense of the statements below also for $\tilde g=0$. For $\tilde g\ge 2$ the field $\kappa(t)$ is known to be uniquely determined by $\widetilde H$, so the involution above is canonical.

 The function $t$ on $\widetilde H$ has two poles, denoted $\infty_\pm$, where we may choose the sign so that $t^{\tilde d}+u$ has a pole of order $ \tilde d$ at $\infty_+$. %; they are the points in $\widetilde H\setminus\widetilde H_a$.

 We denote by $J=J_{\widetilde H}$ the Jacobian variety of $\widetilde H$,  embedding  $\widetilde H$ in $J$ via the map 
 $$
 j:x\mapsto \hbox{class of the divisor $(x)-(\infty_+)$}.
 $$
 
 Often for convenience we shall confound the curve with its embedding in $J$ and divisors with their classes, when there is no risk of misunderstanding.
 
 \medskip
 
 As is well known, each point of $J$ is the sum of $\tilde g$ points on $j(\widetilde H)$. This representation is generally not unique, but if $j(x_1)+\ldots +j(x_{\tilde g})=j(y_1)+\ldots +j(y_{\tilde g})$ then  the fact that  $\widetilde H$ is hyperelliptic is known to imply   that   $\sum (x_i)-\sum (y_i)$ is a divisor of some function in $\C(t)$, hence invariant by the said involution. (See  Lemma \ref{L.unique} for a general version.) 
 
 Inside $J$ we have closed varieties $\widetilde W_m$ defined as the set of sums $j(x_1)+\ldots +j(x_m)$, for $x_i\in\widetilde H$; we have $\dim \widetilde W_m=m$ for $m\le \tilde g$.

 \subsubsection{\tt Pause on the squarefree case} \label{SSS.pause} Before introducing generalized Jacobians, it shall be probably clearer to recall the link with the Jacobian itself and the Pell equation, assuming now that  $D$ is squarefree, i.e. $D_1$ is constant.
 Define then 
 \begin{equation}\label{E.delta}
 \delta:= \hbox{ the class of the divisor $(\infty_-)-(\infty_+)$ in $J$}.
 \end{equation} 
 For instance, we have relations $j(x)+j(x')=\delta$ for every $x\in \widetilde H$, derived by looking at the divisor of the function $t-t(x)$. 
 
 As  mentioned above, the following fact  is classical (attributed to Chebyshev in \cite{Be}):
 \begin{prop}\label{P.pell}
  The Pell equation is solvable if and only if $\delta$ is a torsion point in $J$.
 \end{prop}

 The proof is simple: let $(p,q)$ be a solution of the Pell equation, so $p(t)^2-q(t)^2D(t)=1$ and $p$ is not constant. Then both  $\varphi_\pm :=p\pm qu$ are  rational functions on $\widetilde H$, non constant and regular on the affine part $\widetilde H\setminus\{\infty_\pm\}$. %$\widetilde H_a$. 
 Hence their  divisors of poles are supported at infinity. However % the conjugate function $\varphi':=p-qu$ is also regular at finite points, and 
 $\varphi_+\cdot \varphi_-=1$, hence also the divisors of zeros are  supported at infinity, whence $\div(\varphi_+)=a(\infty_-)+b(\infty_+)$ for integers $a,b$ not both zero. But the degree is zero, so $b=-a$ and $a\delta$ is a principal divisor.  Since $a\neq 0$,  the class of $\delta$ is torsion.
 
 The argument can be reversed: if $a\delta=0$ on $J$, where $a\neq 0$, then $a\delta$ is the divisor of a function $\varphi$, whose divisor is therefore supported at infinity. Then the norm of $\varphi$ down to $\kappa(t)$ has a divisor supported at infinity and hence must be constant. The constant may be taken $1$ by division, whence the result.\footnote{Even on a field not algebraically closed, the constant may be gotten rid of by squaring $\varphi$.} 
 
 Note that this argument also shows that the solutions form a group under the association $(p,q)\mapsto p+qu\in\G_{\rm m}$. This group is either $\Z/2$ or $\Z/2\oplus \Z$; in this case the degree of $p$ in a solution corresponding to $a\delta$  is seen at once to be $|a|$.
 
 \medskip
 
 Even if $\delta$ is not torsion, we may use the above arguments to translate information concerning convergents.  Let $p/q$ be a convergent to $\sqrt D$, for coprime polynomials $p,q$.  Then, after choosing  appropriately the sign related to  $\infty_+$,  we have
 \begin{equation}\label{E.conv}
 \ord_{\infty_+} (p(t)-q(t)u)= \deg q +l,\qquad l>0,
 \end{equation}
for  a positive integer $l$ associated to the convergent, actually the degree of the corresponding partial quotient (in view of \eqref{E.approx}), i.e. $l=\deg a_n$ if $q=q_n$. As we have remarked, if for polynomials $p,q\neq 0$ we have such an equation with $l>0$ then $p/q$ is a convergent. 
 
 Let us set $\varphi:=p-qu$. Note that $\varphi$ has pole divisor supported at infinity, and by \eqref{E.conv} it has a zero at $\infty_+$, hence the divisor of poles is of the shape $a(\infty_-)$ where $a=\deg\varphi$.   On the other hand, because of the zero $\infty_+$ we have $\deg p=\deg q+\tilde d$ and then $a=\ord_{\infty_-}(\varphi)=-\deg p$.

 In conclusion,  we may write
 \begin{equation}\label{E.div}
 \div(\varphi)=(\deg q+l)(\infty_+)+\sigma -(\deg q+\tilde d)(\infty_-)=-(\deg q+\tilde d)\delta +(\sigma -(\tilde d-l)(\infty_+)),
 \end{equation}
where the divisor $\sigma$ is a sum of $\tilde d-l$ points $x_i\in \widetilde H$, not necessarily distinct, but distinct from both $\infty_\pm$ (for otherwise either the zero would be of  higher  order or the pole of lower order).  We  also deduce that for no pair we have $x_i=x_j'$, $i\neq j$,  for otherwise both $p\pm qu$ would vanish at $x_i$ (of order $\ge 2$ if $x_i=x_i'$) and  $p,q$ would not be coprime. \footnote{Similar requirements appear in \cite{Mum}, 3.17.}

 Incidentally, we find back that $l\le\tilde d$. Note also that we may write 
\begin{equation*}
\sigma -(\tilde d-l)(\infty_+)=\sum_{i=1}^{\tilde d-l}((x_i)-(\infty_+)).
\end{equation*}
Reading this equation on $J$ yields
\begin{equation}\label{E.J}
(\deg q+\tilde d)\delta=j(x_1)+\ldots +j(x_{\tilde d-l})\in \widetilde W_{\tilde g-(l-1)}.
\end{equation}
Already this equation shows that the case $l>1$ is very special (we have recalled  above that $\dim \widetilde W_m=m$ for $m\le \tilde g$). 

\medskip

Somewhat conversely, let $m$ be any positive integer, and represent $m\delta\in J$ as a sum $j(x_1)+\ldots +j(x_{\tilde g})$ of $\tilde g$ points of $\widetilde H$. Then $(m-\tilde g)(\infty_+)+(x_1)+\ldots +(x_{\tilde g})-m(\infty_-)$ is the  divisor of some function, necessarily of the shape $p^*(t)-q^*(t)u$, for polynomials $p^*,q^*$. We then find that $p^*/q^*$ is a convergent; however $p^*,q^*$ may not be coprime: this corresponds to the fact that we may have some pairs $x,x'$ among the $x_i$, in which case the representation  could be reduced to less that $\tilde g$ summands (on decreasing $m$).  We can also have some $x_i=\infty_+$ (in which case the order of zero increases) or $x_i=\infty_-$ (in which case the representation `comes' from a similar one with smaller $m$ and less that $\tilde g$ summands). 

\medskip

This is a viewpoint on Pad\'e approximations to $\sqrt{D(t)}$ different from the more usual one involving linear algebra. (See also \cite{Be}.) It may lead to algorithms in various directions (e.g. in computing torsion orders). 

\medskip

All of this says that the convergents correspond to expressing multiples of $\delta$ as sums of $\tilde g$ points of $\widetilde H$ in $J$. %The sum-map from the $\tilde g$-th symmetric power to $\tilde J$ is a birational morphism, and so we may find
For instance, when $\tilde g=1$ we have just to find $m\delta$ as a point on an elliptic curve, by the well-known procedures. This also yields certain recurrence formulae on which  we do not pause  here (but see Example \ref{EX.ell}).  %This further allows to express $\varphi$ in terms of  theta functions:  e.g., identifying $\infty_+$ with the origin in a torus and $\infty_-$ with a point $z_0$, we can write  $\varphi=(\sigma(z)/\sigma(z-z_0))^{m}\sigma(z-mz_0)$ ($m=\deg q+d$) for $z$ in the torus, corresponding to $t$, as in \cite{[L2]}, p. 242.

%\medskip

\begin{small}
\begin{rem}{\tt Heights of convergents}. \label{R.BoCo} To  conclude  this pause, let us see how these  facts imply the behaviour of heights mentioned above in the Introduction, where we suppose now that $\kappa$ is a number field. Namely, we prove the inequality 
$$
h(q)\gg (\deg q)^2
$$
 for the convergents $q(t)$ associated to the non-Pellian $\widetilde D$. We have seen in the proposition above that $\widetilde D$ is Pellian if and only if $\delta$ is torsion in $J$. Suppose   this does not hold. Then $\hat h(\delta)>0$, where $\hat h$ denotes a canonical height on $J$, and by standard facts (see \cite{[BG]}) we have $\hat h\left(j(x_1)+\ldots +j(x_{\tilde d-l})\right)=(\deg q+\tilde d)^2\hat h(\delta)\gg (\deg q)^2$. Since the height is  a quadratic form, we deduce that $\max\hat h(j(x_i))\gg (\deg q)^2$, whence the same lower bound holds for $\max h(x_i)$, for any height $h$ on $\widetilde H$ associated to an ample divisor. But the values $t(x_i)$ are roots of the polynomial $p(t)^2-q(t)^2\widetilde D(t)$, of degree $\tilde d-l$. We conclude that the height of this polynomial has the same kind of lower bound, and this must hold as well for both $h(p), h(q)$ (since $q(t)$ determines $p(t)$ linearly with coefficients of height $\ll\deg q$).\footnote{It may happen that $\widetilde D(t)$ is Pellian but $D(t)$ is not; in this case the height of the $q_n$ grows linearly in $n$. This may be proved from the considerations below, which this time relate with heights in a torus $\G_{\rm m}$ rather than  an abelian variety. }

The same arguments also show the converse bound $h(q)\ll (\deg q)^2$. Actually, this also follows from Siegel's lemma, since the $m$-th coefficient of the  Laurent series for $\sqrt{D(t)}$ has height $\ll m$. (In the Pellian case we have $h(q_n)\ll \deg q_n$.) As already remarked, the lower bound was discovered by Bombieri and P.B. Cohen and proved in \cite{Bo-Co} in rather greater generality. 
\end{rem}

\begin{rem}{\tt Values of convergents}. \label{R.values} The large height of the convergents and of the $x_i$ makes it also difficult to detect the behaviour of values $q_n(\xi)$ at a given point $\xi$. Note that this could  be useful e.g. for deriving numerical approximations to $\sqrt{D(\xi)}$ on plugging in $t=\xi$ in the Pad\'e approximation, suitably normalized.  The large height may however destroy  the information.   Also, for growing degrees $\approx n$ of the convergents, a given $\xi$ a priori could go very near to some of the $t(x_i)$, again confounding the expectations. As we have seen, these $x_i$ are essentially functions of $n\delta$. A deep theorem of Faltings prevents the distance $|t(x_i)-\xi|$ to be less than $\exp(-\epsilon \hat h(n\delta))$ (with respect to any given absolute value). However since the height behaves quadratically this is   too weak to locate $q(\xi)$. \footnote{One exception occurs in the elliptic case, when  lower bounds  of Masser   for linear forms in elliptic logarithms should suffice.} 
\end{rem}
\end{small}

\subsubsection{\tt Generalized Jacobians} \label{SSS.gj} After this pause, we go to the general case. Now, if $D(t)$ is not squarefree the curve $u^2=D(t)$ is singular also at finite points. We can however extend much of the previous considerations by using generalized Jacobians, for which we refer to Serre's book \cite{SeAGCF}, see especially Chs. IV, V and VII.
\footnote{We warn the reader that to avoid a somewhat complicated notation sometimes  one may prefer, at least for part of the issues,  to   think of the case when $D_1$ has no multiple roots and is prime to $\widetilde D$ or even to stick to the squarefree case just considered.} % the case when it is constant, which is what we have done in \S \ref{SSS.pause}.

\medskip

Let then $\rho$ be a root of $D_1(t)$ of multiplicity $e=e_\rho\ge 1$. There are two cases to consider:

\medskip

{\bf Case 1}.  $\widetilde D(\rho)\neq 0$. In this case there are two points $\xi_\rho,\xi_\rho'$ of $\widetilde H$ above $t=\rho$. The total multiplicity of $\rho$ as a root of $D(t)$ is $2e$.

{\bf Case 2}. $\widetilde D(\rho)=0$, so there is a single point $\xi_\rho$ of $\widetilde H$ above $t=\rho$ (which is ramified with respect to $t:\widetilde H\to\P_1$, and we have $\xi_\rho'=\xi_\rho$). The total multiplicity of $\rho$ as a root of $D(t)$ is $2e+1$. 

\medskip

We consider the {\it strong equivalence}  of divisors of degree $0$ on $\widetilde H$ with support disjoint from the set $\cS$ of all such points $\xi_\rho,\xi_\rho'$ (we also say `coprime' to $\cS$), defined by saying that
\begin{equation}\label{E.strong}
A\approx 0
\end{equation}
precisely if $A$ is principal as a divisor on $\widetilde H$, and $A=\div(f)$, where $f-1$  vanishes at both $\xi_\rho$, $\xi_\rho'$ in Case $1$ (resp. at $\xi_\rho$ in Case $2$) to order $\ge e$ (resp. $\ge 2e+1$).

It is proved in \cite{SeAGCF}  (see especially Ch. IV) that this last condition makes the set of divisors of degree $0$ coprime to $\cS$ a (commutative) group-variety which is an extension of the usual Jacobian $J$ of $\widetilde H$ by a linear group $\Lambda=\Lambda_{\tt m}$ which is a product of a power of $\G_{\rm m}$ by a power of $\G_a$.  More precisely, this extension is associated to the {\it modulus} ${\tt m}=\sum_s\epsilon_s\cdot s$, where $\epsilon_s=e_\rho$ if $s=\xi_\rho,\xi_\rho'$  in Case 1 and $=2e_\rho+1$ in Case 2, and is denoted $J_{\tt m}$. As explained in \cite{SeAGCF}, if $\tt m\neq 0$ we have an exact sequence
\begin{equation}\label{E.J_m}
0\to\Lambda \to J_{\tt m}\to J\to 0,
\end{equation} 
where $\Lambda=\G_{\rm m}^{|\cS|-1}\times \G_a^{\sum_s(\epsilon_s-1)}$; the association is explained in detail in the quoted book.
Of course the map on the right is obtained by weakening the strong equivalence above to usual linear equivalence. 

We shall actually need a group-variety smaller than this. It is defined by taking the quotient of $J_{\tt m}$ by the group of strong classes of principal divisors $A$ prime to $\cS$, such that $A\approx A'$ (where $A\mapsto A'$ is the usual involution); so this is a subgroup of $\Lambda$. It is readily checked that this is well-defined  and that  the quotient  group is isomorphic to an extension of $J$ by a product $\prod_{D_1(\rho)=0}L_\rho$, where the group $L_\rho$ is $\G_{\rm m}\times \G_a^{e-1}$ in Case $1$ and $\G_a^{e}$ in Case $2$.

\medskip

Observe that the principal divisor classes factored out correspond to  functions  $f=a(t)+b(t)u\in \C(\widetilde H)$ with rational functions $a,b\in \bar \kappa(t)$ such that $a$ has no poles or zeros in $\cS$  and $b$ is divisible by $D_1$. In practice, we are detecting the individual values of  ratios $f/f'$  at the points in $\cS$,  actually taking into account the expansions up to the multiplicities. (Note that at pairs $\xi_\rho,\xi_\rho'$ these values are reciprocal; this is why we have a single copy of $\G_{\rm m}$ for each pair and the dimensions are all  halved.)

We denote by $G=G({\tt m})$ such a group-variety, so we have an exact sequence of algebraic groups
\begin{equation}\label{E.G}
0\to  \prod_{D_1(\rho)=0}L_\rho\to G\stackrel{\pi}\to J\to 0.
\end{equation} 
Hence the dimension   of $G$ is 
$$
g:=\dim G=\dim J+\deg D_1=\tilde g+\deg D_1=\tilde d-1+\deg D_1=d-1.
$$
Naturally, $g$ is the arithmetic genus of the singular curve defined by $u^2=D(t)$ at finite points, and smooth at infinity.\footnote{This group may be also seen as a fiber product over $J$ of the various extensions obtained at the individual  roots $\rho$.} 

\medskip

As in \cite{SeAGCF}, we have an embedding of $\widetilde H\setminus \cS$ in $G$, obtained similarly to the one in $J$, i.e. by sending a point $x\in\widetilde H\setminus \cS$ first to class in $J_{\tt m}$ of  the divisor $(x)-(\infty_+)$ and then taking the image of this class in $G$, which we denote with $[x]$. However if $\tt m\neq 0$  the map is not a morphism on all of $\widetilde H$. 

We define $W_h=W_h(\tt m)$ as the  image of the map $(x_1,\ldots ,x_h)\mapsto [x_1]+\ldots +[x_h]$ from the symmetric $h$-th power of $\widetilde H\setminus \cS$ to $G$.  It is a `constructible' set, by a well known theorem of Chevalley; however it may be not Zariski-closed (except in the case of the usual Jacobian, i.e. when $\cS$ is empty) and then we let $\overline{W_h(\tt m)}$ be its Zariski closure. 

\medskip

 It may be easily checked that  %{\it generically} any divisor class in $G$ is the   sum of $g$  elements of the shape $[x]$ so defined (for $x\in \tilde H\setminus \cS$). 
 $\overline{W_g}=G$, and that actually this map is a birational   isomorphism (see \cite{SeAGCF}). %betwee $G$ and the symmetric power $\tilde H^{(g)}$ of $\tilde H$.
For $h<g$ we must have $\dim W_h=h$ and  we obtain proper subvarieties of $G$.\footnote{At least in the case of the usual Jacobian, these subvarieties have been widely studied in the context of special divisors and linear series. See e.g. \cite{ACGH}, where a somewhat different notation is used; indeed, our notion depends on the embedding of $H$, which in other contexts may be inconvenient. See  also \cite{DF} and \cite{F}, where these varieties appear in connection with rational points of bounded degree, on which we shall further comment.}

 \medskip
 
 Note also that if we have an equality $\sum_{i=1}^g[x_i]=\sum_{j=1}^g[y_j]$  (for points not in $\cS$) then there exists a function $f$ on $\widetilde H$ with divisor $\sum (x_i)-\sum (y_j)$ such that $\div(f/f')$ is strongly equivalent to $0$. This easily entails % (e.g. on using the above description of such functions) 
 that $f\in \C(t)$, so the $x_i$ which are not $\infty_+$,  or among the $y_j$, must  appear together with $x_i'$ and similarly for the $y_j$. Indeed, we have the following simple lemma, useful throughout:
 
 \begin{lem}\label{L.unique}
 Notation as above, let  $f=(a(t)+b(t)D_1(t)u)/c(t)\in\C(\widetilde H)$ where $a,bD_1,c$ are coprime polynomials in $\C[t]$. % and $c,D_1$ are coprime. 
 Then either $\deg f\ge d$ or $b=0$.
 \end{lem}

\begin{proof}  %One may assume that $\kappa=\bar \kappa$. 
Let  $\xi\in \C$ and let $m=\ord_\xi c(t)>0$. Suppose first that $\widetilde D(\xi)\neq 0$ and observe that there are two points in $\widetilde H$ above $t=\xi$ and that at least one is a pole of $f$ with multiplicity $m$ (for otherwise $\xi$ would be a zero of both $b(t)D_1(t)$ and $a(t)$). If $\widetilde D(\xi)=0$,  there is a unique point of  $\widetilde H$ above $t=\xi$, and (for the same reason) this must be   a pole of $f$ with multiplicity at least $2m-1$. Observe that these poles contribute at least $\deg c$ to $\deg(f)$. If $\deg c\ge d$ we are done; otherwise,   if $b(t)\neq 0$ then at least one between $\infty_\pm$ is a pole of $f$ with order at least $d-\deg c$, concluding the argument.
\end{proof}

%To go ahead, for each root $\rho$ of $D_1(t)$ which falls in Case 1 we can form the modulus $(\xi_\rho)+(\xi_\rho')$ and consider a corresponding generalized Jacobian, which shall be an extension of $\widetilde J$ by $\G_{\rm m}$. It is very easily checked that there is a surjective homomorphism from $G$ to this group-variety. 

 \medskip

Finally, if $\kappa$ is a field of definition for the curve and the points in $\cS$, these varieties and maps are defined over $\kappa$. We do not pause instead on the %somewhat delicate    
  question of when these group-extensions split as products. 

\subsubsection{\tt A `canonical' algebraic subgroup}\label{SSS.can}   We have seen that at least in the squarefree case the Pell equation is solvable precisely when $\delta$ is torsion in the Jacobian. Even if this does not happen, the multiples of $\delta$ are especially relevant in the context. 
Hence,  for a modulus $\tt m$ as above, let us define the `canonical' algebraic subgroup $\Delta({\tt m})\subset G({\tt m})$ as 

\medskip

{\it $\Delta({\tt m})=$  the Zariski closure in $G({\tt m})$ of the set of multiples of the (class of) $\delta$. }

\medskip 

We shall also usually denote by $\Delta_0(\tt m)$ the connected component of identity in $\Delta(\tt m)$. 

For instance, in the squarefree case  we have $\tt m=0$ and $\Delta_0:=\Delta_0(0)$ is an abelian subvariety of $J$, and hence if  $J$ is simple, as generically happens,   then either the Pell equation is solvable or $\Delta_0=J$ which  yields relevant consequences, as we shall see.

%%% QUI

\subsubsection{\tt Convergents and divisors}\label{SSS.pause2} 

We now give some analogues of the  facts and formulas previously obtained for the squarefree case, omitting the proofs because completely similar. 

\medskip

We let $u_1:=D_1(t)u$, so $u_1^2=D(t)$. 
Also, we continue to denote $\delta:=(\infty_-)-(\infty_+)$ and use the same notation for its  image   in $G$, i.e. $\delta=[\infty_-]$. 

The solvability of the Pell equation for $D(t)$  now  corresponds to the fact that $\delta$ is torsion on $G$. Namely, with exactly the same proof as above, we have

 \begin{prop}\label{P.pell2}
  The Pell equation for $D(t)$  is solvable if and only if $\delta$ has finite order in $G$, i.e. $\Delta$ is finite.
 \end{prop}

In general, as before let $p/q$ be a convergent to $\sqrt D$, for coprime polynomials $p,q$ and let as above 
 \begin{equation}\label{E.conv2}
 \ord_{\infty_+} (p(t)-q(t)u_1)=\deg q +l,
 \end{equation}
 where $l>0$.  
 Let us set $\varphi:=p-qu_1$. We can repeat part of the above considerations, and conclude that  $\deg p=\deg q+d$ and 
 \begin{equation}\label{E.div2}
 \div(\varphi)=(\deg q+l)(\infty_+)+\sigma -\deg p\cdot  (\infty_-)=-\deg p\cdot \delta +(\sigma -(d-l)(\infty_+)),
 \end{equation}
where the divisor $\sigma$ is a sum of $d-l$ points $x_i\in \widetilde H$, not necessarily distinct, but distinct from both $\infty_\pm$.  

A difference with the previous case is that  we now can  deduce that for no pair we have $x_i=x_j'$, $i\neq j$ only  if $p,D_1$ are coprime.

\medskip

We find back again  that $l\le d$. 

We cannot in general read  this equation on $G$, since $p,D_1$ may be not coprime. We shall   reduce later to the coprime case. But we can still read it on $J$, which gives
\begin{equation}\label{E.J2}
(\deg p)\delta= (\deg q+ d)\delta=j(x_1)+\ldots +j(x_{d-l}).
\end{equation}

\subsection{\tt Some formulae for convergents}\label{SS.formulae}  We let $(p_n,q_n)$ be the sequence of convergents to $\sqrt{D(t)}$, and let $a_n$ be the partial quotients, putting $l_n:=\deg a_n$.  We give some formulae which shall be applied later  (some of which may be also found in \cite{[vdPT]}). 

Taking into account the notation above, we also set $\varphi_n:=p_n-q_nu_1$, where as before $u_1=\sqrt D=D_1u$. 

From the formulae $p_nq_{n+1}-p_{n+1}q_n=(-1)^n$ we derive
\begin{equation*}
\varphi_n\varphi_{n+1}'=p_np_{n+1}-q_nq_{n+1}D+(-1)^nu_1=S_n+(-1)^nu_1,
\end{equation*}
where  $S_n(t):=p_np_{n+1}-q_nq_{n+1}D$.  For instance, $S_0=p_0p_1=a_0$. 

Let also $R_n(t):=\varphi_n\varphi_n'$ be  the norm of $\varphi_n$ down to $\kappa(t)$, so $R_n$ is a polynomial; its roots are the values $t(x_i)$, the $x_i=x_{in}$ coming from formula \eqref{E.div2} above with $(p,q)=(p_n,q_n)$, and $\deg R_n=d-l_n$. Taking norms of the last displayed equation, we get 
\begin{equation*}
R_n(t)R_{n+1}(t)=S_n(t)^2-D(t),
\end{equation*}
whence in particular 
\begin{equation*}
\deg (S_n^2-D)=2d-l_n-l_{n+1}\le 2d-2,
\end{equation*}
so $S_n=\pm \sqrt D+O(t^{d-l_n-l_{n+1}})$, which implies  $S_n=\pm t^d +O(t^{d-2})$. 

\medskip

We have the standard recurrence formulae $p_{n+1}=a_np_n+p_{n-1}$,  $q_{n+1}=a_nq_n+q_{n-1}$, $n\ge 0$, which yield in particular $\deg q_{n+1}=\deg q_n+l_n$ and $\varphi_{n+1}=a_n\varphi_n+\varphi_{n-1}$.

%%%%%%%%%%%%%%%%%%%%%%%%%%%%%%%%%%%%%
\begin{comment}
From \eqref{E.div2} we also find, on  putting $\ord=\ord_{\infty_+}$, 
\begin{equation*}
\ord(\varphi_n\varphi_{n+1}')=\deg q_n+l_n-\deg q_{n+1}-d=-d,
\end{equation*}
\begin{equation*}
\ord(\varphi_n'\varphi_{n+1})=-\deg q_n-d+\deg q_{n+1}+l_{n+1}=-d+l_n+l_{n+1}.
\end{equation*}

From the above we derive
\begin{equation*}
\varphi_n\varphi_{n+1}'-\varphi_n'\varphi_{n+1}=2 (-1)^nu_1,
\end{equation*}
\end{comment} 
%%%%%%%%%%%%%%%%%%%%%%%%%%%%%%%%%%%%%

Setting also $\nu_n:=\varphi_{n+1}/\varphi_n$, we obtain $\nu_n\nu_n'=R_{n+1}/R_n$ and 
\begin{equation*}
\nu_n={\varphi_{n+1}\varphi_n'\over R_n}={S_n+(-1)^{n+1}u_1\over R_n}.
\end{equation*}

On the other hand, the recurrence for $\varphi_n$ yields 
\begin{equation*}
\nu_n=a_n+{1\over \nu_{n-1}}.
\end{equation*}

Conjugating this formula and adding, we get
\begin{equation*}
2{S_n\over R_n}=\nu_n+\nu_n'=2a_n+{\nu_{n-1}+\nu_{n-1}'\over \nu_{n-1}\nu_{n-1}'}=2a_n+2{S_{n-1}\over R_n},
\end{equation*}
and finally
\begin{equation}\label{E.a_n}
a_n={S_n-S_{n-1}\over R_n}. 
\end{equation}
Comparing degrees, we see that $\deg(S_n-S_{n-1})=d$, whence $S_n=(-1)^n\sqrt D +O(t^{d-l_n-l_{n+1}})$.  In particular,
\begin{equation*}
a_n=2(-1)^n{\sqrt D\over R_n} + O(t^{-1}).
\end{equation*}
This also exhibits $a_n$ as the polynomial part of $2(-1)^na_0/R_n$, so we can calculate inductively  these quantities e.g. in the order $\ldots \to R_n\to a_n\to S_n\to R_{n+1}\to\ldots$. 

\medskip

Recall now that we are assuming that $D(t)=t^{2d}+O(t^{2d-2})$, so $\sqrt D=t^d+O(t^{d-2})$. 

Also,  omitting the index $n$ for a moment, the roots of $R(t)=R_n(t)$ are the $t_i=t(x_i)$, i.e. $R(t)=c\prod_{i=1}^{d-l}(t-t_i)$, $c=c_n$.  We find therefore for example that
\begin{equation*}
a_n=(-1)^n{2\over c}\left(t^l+(\sum t_i)t^{l-1}+O(t^{l-2})\right).
\end{equation*}

\section{A Skolem-Mahler-Lech Theorem for Algebraic Groups} \label{S.SML} 

The Skolem-Mahler-Lech Theorem (SML in the sequel) states that for a linear recurrence sequence $(u_n)_{n\in\N}$ (over $\C$) the set of $n$ with $u_n=0$ is the union of a finite set and a finite set of arithmetical progressions. 
Taking into account the structure of linear recurrences, we are simply describing the set of integral zeros of an exponential polynomial $\sum_{i=1}^rP_i(n)a_i^n$ for complex polynomials $P_i$ and complex numbers $a_i\neq 0$. 

%We may view this as 
This is an algebraic relation %holding for  
on the points $\gamma_n:=(n,a_1^n,...,a_r^n)$; on the other hand, $\gamma_n$ is just $n$-times $\gamma_1$ in the algebraic group $\G_a\times \G_{\rm m}^r$. In this view, a natural generalization is obtained by taking an algebraic group $\Gamma$ (over a subfield of $\C$), a point $\gamma\in \Gamma$, and asking about the Zariski closure of an arbitrary  set of multiples  (powers) $\gamma^n$ in $\Gamma$. 

\medskip

To present such a generalization, to be applied later to our context,  is the task of the present short section. These results, though perhaps somewhat implicit in the context of the SML theorem, %but to our knowledge they 
seem to have been explicitly stated for (one of) the first time(s)   in the 2009 book  \cite{[Z5]} by the writer, with a  sketch of a fairly simple proof (based on ideas - mostly of Skolem and Chabauty -   near to the original proofs of SML). This has never appeared in articles and we intend to insert here a more precise version of  such short proof, with the addition of a relevant corollary,  for clarity and completeness. %to make the paper more self-contained.

We mention that the (recent)  literature contains other   versions of the SML theorem; however most of them, though covering several other situations,  do  not to apply generally to our context,   one exception occurring  within the 2010 paper \cite{BGKT}, where a SML Thm. is obtained concerning iterates of arbitrary  \'etale maps. Also, theorems of Faltings and others (used here for the proof of Theorem \ref{T.irr}) would suffice for several of the applications we have in mind. However for the above reasons we prefer to insert our simple  and very short treatment, which moreover yields sometimes supplementary  information (e.g. of effective nature).

\medskip

Let then $\Gamma$ be an algebraic group over  $\C$, and $\gamma\in \Gamma$. %We denote by $[n]:\Gamma\to\Gamma$ the map of multiplication by $n$. 
We start with a simple lemma.

\begin{lem}\label{L.SML}
 For $b\in\Z$, let $Z(b)$ be the Zariski-closure (in $\Gamma$) of the set 
$\{\gamma^{nb} :n\in\N\}$, setting $Z=Z(1)$. Then we have:  

(i) $Z(b)$  is a commutative algebraic subgroup of $\Gamma$. 

(ii) The connected component $Z_0$ of the identity in $Z$ equals $Z(\mu)$  for  some 
integer $\mu$. 

(iii) For $b\neq 0$,  $Z(b)$ is a finite union of cosets of $Z_0$. 
\end{lem} 

\begin{proof} Let $z\in Z$. If $X$ is a closed subset containing all  multiples $\gamma^n$ ($n\in\N$) then $\gamma^{-1}X$ also has this property.  Therefore it contains $z$, whence $\gamma z\in X$ and hence $\gamma z\in Z$. It follows easily that $Z$ is closed for multiplication. Further, if a closed set $X$ contains all large multiples $\gamma^n$, then $\gamma^{-h}X$ contains them all for some $h>0$, whence it contains $Z$ and by what has been proved $X$ itself must contain $Z$. It follows that $Z$ is an algebraic subgroup of $\Gamma$, and by similar arguments it follows that it is  commutative. Replacing $\gamma$ by $\gamma^b$ we obtain (i). 

By general (easy) theory, we can write $Z$ as a finite union of cosets of $Z_0$. Multiplication by $\gamma$ permutes these cosets and hence some positive power of $\gamma$ lies in $Z_0$, and let $\gamma^\mu$ be the minimal such  power. Then $Z(\mu)$ is contained  in $Z_0$, and $Z$ is the union of  the finitely many translates of $Z(\mu)$ by the powers $\gamma^\nu$, $0\le \nu<\mu$, whence $Z_0=Z(\mu)$ by minimality, proving (ii). 

Finally, a suitable  finite union of cosets of $Z(b)$ certainly contains $Z$, whence (iii).
\end{proof}

Note that the lemma shows in particular that  it does not matter if we start with all multiples $\gamma^n, n\in\Z$ or merely with those with $n\in\N$.

\medskip

Now, as mentioned above, the question arises of what can be said about the Zariski-closure of a subset of   all the
powers of $\gamma$, namely of a set $\{\gamma^{a_n}, n\in\N\}$ where $(a_n)_{n\in\N}$ is a sequence of  (distinct) integers. We have the following

\begin{thm}\label{T.SML}
Let  $\Gamma$ be an algebraic group over $\C$,   let $\gamma\in\Gamma$ and let $(a_n)_{n\in\N}$ be a sequence of integers. The Zariski-closure of   $\{\gamma^{a_n}: n\in\N\}$ is a finite union of points and cosets of the connected component of the identity  of the Zariski-closure of  $\{\gamma^n:n\in\Z\}$. 
 \end{thm}

 \begin{proof}   By Lemma \ref{L.SML}  (i), we can replace $\Gamma$ with the algebraic group  denoted above $Z$, which is 
commutative, so we use
from now on an additive notation. Further, by partitioning $Z$ into (finitely many) cosets of $Z_0$, we
may assume, on replacing $\gamma$ with a suitable power of it,  that  $Z=Z_0$ is connected. We prove that if
$\{a_n\}$ is infinite then $\{a_n\gamma\}$ is Zariski-dense in $Z$; this plainly leads at once to the theorem.

Then suppose by contradiction that there is a rational  nonconstant function $f$ on $Z$, defined at
the  points $a_n\gamma$ and such that $f(a_n\gamma)=0$ for all $n$. 

 Now, $Z,\gamma$ and $f$ are  defined over a finitely generated subfield  of $\C$, and it is well known
that this may be embedded in some finite extension $\kappa$ of  a field $\Q_p$ (see \cite[page 61]{SeMWT}).

Let $\O$ be the valuation ring of $\kappa$;  by \cite[Corollary 4 to Theorem 2, page 151]{SeLALG}, 
 $Z(\kappa)$ has an open subgroup $H$ analytically isomorphic to $\O^d$, where $d=\dim Z$. 
 
 By taking $p$ very
large, we may assume that $Z,\gamma$ have good reduction at $p$. Since the residue field of $\kappa$ is finite, it
follows that a suitable multiple  $l\gamma$ lies in $H$. Then, 
  by partitioning  $(a_n)$ into a finite number of subsequences according to the class of $a_n$ modulo $l$,  we may assume that the $a_n$ are pairwise congruent
modulo $l$, so we may write $a_n=c+b_nl$ with a fixed integer $c$ and integers $b_n$.  

Through  the (analytic) isomorphism
$H\cong \O^d$, the element  $l\gamma\in H$ and the function  $f(c+x)$ become resp.   $\xi\in\O^d$ and a locally   analytic function    $\phi$ on $\O^d$
such that $\phi(b_n\xi)=0$ for all $n$. This function    induces a locally analytic function $z\mapsto \phi(z\xi)$ on the
compact set $\O$  with   infinitely many zeros therein, so it must vanish identically. But then $\phi(n\xi)=0$ for
all integers $n$, whence $f((c+n)l\gamma)=0$ for all $n$, and we have a contradiction because $\{nl\gamma: n\in\Z\}$ is
Zariski-dense ({\it e.g.\/} on recalling the previous  lemma). 
 \end{proof}

In concrete situations, this proof may lead to effectivity in various shapes; for instance, for the case of the original SML, it sometimes leads to the actual determination of all the zeros of a recurrence. This may depend on a careful choice of the prime $p$ appearing in the arguments. \footnote{See M. Stoll's recent paper \cite{St} for some definite progress in this direction.} 
 This choice often  leads to an effective upper bound  for the number of zeros. Similar supplementary information may come in other applications.

\medskip

We conclude this short section with a  corollary, useful for us. Recall that a constructible subset of an algebraic variety is an element of the Boolean algebra generated by the Zariski-closed subsets. 
With the previous notation we have:

\begin{cor}\label{C.SML}  Let $U$ be  a constructible set in $\Gamma$ and  let $K$ be the set of integers $k$ such that $\gamma^k\in U$. Then $K$ is a finite union of arithmetical progressions, modulo the integer $\mu$ of Lemma \ref{L.SML}(ii), 
plus and minus  finite sets.
\end{cor}

A proof is  readily obtained from the theorem.  Indeed, we can  replace $\Gamma$ with $Z$ and $U$ with $U\cap Z$. By the lemma, the components of $Z$ are of the shape $\gamma^cZ_0$ and it suffices further to work with the intersections of $U$ with each component.  Replacing $U$  with $\gamma^{-c}U$ we may finally work with $Z_0$ in place of $Z$. Now, if $U$ is contained in a proper closed subset of $Z$ then the set of  powers  of $\gamma$  in $U$ must be finite by the theorem. Otherwise, $U$ contains $Z_0\setminus U_1$, where $U_1$ is a proper closed subset of $Z_0$; again, $U_1$ can contain only finitely many powers  of $\gamma$, whereas $Z_0$ contains all powers of $\gamma^\mu$  and no other powers (which are contained in the other components), concluding the argument.

%%%%%%%%%%%%%%%%%%%%%%%%%%%%%%%%%%%%%%%%%%%%%%%%%%%%%%%  FINE SKOLEM ML  %%%%%%%%%%%%%%%%%%%%%%%
%%%%%%%%%%%%%%%%%%%%%%%%%%%%%%%%%%%%%%%%%%%%

 %%%%%%%%%%%%%%%%%%%%%%%%
 %%%%%%%%%  INIZIO Proofs %%%%%%%%%%%%%%%%%%%%
 %%%%%%%%%%%%%%%%%%%%%%%%%%%

 \section{Proof of main assertions} \label{S.proof} 
 
 \subsection{\tt General deductions} \label{SS.proof1} We shall begin with some general deductions relevant in themselves and useful for several of the results. We shall often abbreviate $\ord :=\ord_{\infty_+}$. 
 
 \medskip
 
 To start with, let us consider a convergent $(p,q)$ to $\sqrt D$ and rewrite for convenience a previous formula involving $\varphi:=p-qu_1=p-qD_1u$:
  \begin{equation*}
 \div(\varphi)=-(\deg q+d)\delta +(\sigma -(d-l)(\infty_+)),
 \end{equation*}
where the divisor $\sigma$ is a sum of $d-l$ points $x_i\in \widetilde H$, not necessarily distinct, but distinct from both $\infty_\pm$.  Also, $l$ is the degree of the corresponding partial quotient. 

\medskip

 Now, a small complication comes from the fact that $p, D_1$ may not be coprime. Let then $r(t)$ be their (monic) $\gcd$, so that $p=rp^*, D_1=rD_1^*$ and $\varphi=r(p^*-qD_1^*u)=r\varphi^*$. 
 
 Of course this depends on the particular convergent, but at least we have only finitely many choices for $r(t)$. Also, we have a corresponding modulus $\tt m^*$  (obtained by considering $D_1^*$ in place of $D_1$) and generalized Jacobian $G^*:=G(\tt m^*)$ and canonical algebraic subgroup $\Delta^*:=\Delta(\tt m^*)$  (as in \S \ref{SSS.can}). They also have  only finitely many possibilities (i.e. dependent only on $D$), and there are obvious surjective homomorphisms from $G(\tt m)$ to $G(\tt m^*)$.

 Let also $r^*=\deg r$, $d^*=d-r^*$. Then the formula now leads to
 
 \begin{equation}\label{E.div*}
 \div(\varphi^*)=-(\deg q+d^*)\delta +(\sigma^* -(d^*-l-r^*)(\infty_+)).
 \end{equation}

 \begin{small}
 \begin{rem}\label{R.comes} We note  in passing that this corresponds to the fact that this convergent $(p,q)$ comes from a convergent $(p^*,q)$ to $\sqrt{D^*}$  (where $D^*=(D_1^*)^2\widetilde D$), and that the order of the approximation has improved by $r^*$: in fact, $\ord (\varphi^*)=\ord (\varphi)+r^*$. So, in particular we see that this phenomenon must be `rare' and `usually' $p,D_1$ should be coprime.
 \end{rem}
 \end{small}

 As to the divisor $\sigma^*$, this time it is a sum of $d^*-l-r^*$ points $(x_i)$ % in $\tilde H\setminus \cS^*$ (where $\cS^*$ is defined from $D_1^*$ as $\cS$ was defined from $D_1$); also,  
 each of them  different from both $\infty_\pm$.  In particular, $d^*\ge l+r^*$, i.e. $d\ge l+2r^*$. 
 
 Also, since $p^*,D_1^*$ now are coprime, the $x_i$ cannot appear in $\tt m^*$, i.e. $\sigma^*$ is coprime with $\tt m^*$. This is very useful: it implies first that we can consider divisor classes in $G(\tt m^*)$, and also  that no pair $x_i,x_j$ for $i\neq j$ may be conjugate under the involution, for otherwise $p^*,q$ would not be coprime. 
 
Observe that the divisor of $\varphi^*$ is prime to $\tt m^*$ and that $\varphi^*/(\varphi^*)'=\varphi/\varphi'$  is congruent to $1$ relative to $\tt m^*$; hence $\div(\varphi^*)$ vanishes in $G(\tt m^*)$, whence taking divisor classes of \eqref{E.div*}  in $G(\tt m^*)$ we obtain
\begin{equation}\label{E.class*}
(\deg p^*)\delta=(\deg q+d^*)\delta=\sum_{i=1}^{d^*-l-r^*}[x_i]\qquad \hbox{in $G(\tt m^*)$}.
\end{equation}
 
 In particular, the multiple of (the class of) $\delta$  on the left hand side belongs to the constructible subvariety of $G(\tt m^*)$ denoted $W_{d^*-l-r^*}(\tt m^*)$ in \S \ref{SSS.gj} above. Then, recalling  that $G(\tt m^*)$ has dimension $d^*-1$  and that $\dim W_h=h$,  we see  that this equation reflects something unusual if  $r^*+l>1$). 
 
 \medskip
 
 It is very  important to note that these considerations may be  essentially  reversed. If we have \eqref{E.class*}, with an integer $k$ in place of $\deg p^*=\deg q+d^*$,  then by definition  there is a rational function $f$ on $\widetilde H$ whose divisor is given by the right hand side of \eqref{E.div*} and such that $f/f'$ is congruent to $1$ modulo $\tt m^*$.  Hence we may certainly write $f=a(t)-b(t)D_1^*(t)u$ with polynomials $a,b$. 
 
 Let us assume also the above conditions on the $x_i$: none of them is $\infty_\pm$ and if $i\neq j$ we have $x_i\neq x_j'$.  Then it follows that  $f$ has a zero  of order $k-d^*+r^*+l$ at $\infty_+$ and a pole of order $k$ at  $\infty_-$, we see that $\deg a(t)=k$, $\deg b(t)=k-d^*$, and certainly $a/b$ is a convergent to $\sqrt{D^*}$. Actually,  $a,b$ must be coprime and $(a,b)$ is a  convergent (as a pair)  up to a constant; the degree of the corresponding partial quotient shall be $l+r^*$. 
 
 If we allow some $x_i$ to be $\infty_+$, then   the corresponding $[x_i]=0$ and we may remove them, increasing correspondingly $l$.
 
  If we allow some $x_i=\infty_-$, then the corresponding $[x_i]=\delta$ and we may subtract it from both sides,  decreasing $k$  by $1$ and  increasing $l$ by $1$. This shall produce a smaller degree of $a(t)$. 
  
  Finally, if we allow equations $x_i=x_j'$ for some pairs $i\neq j$, then grouping these pairs we shall obtain divisors of polynomials in $t$,  and simply $a,b$ shall not be coprime;  dividing out by a $\gcd$, say of degree $c<d$, we shall obtain another equation of type \eqref{E.class*} but with a smaller value $k-c$ in place of $k$,   and a larger one $l+2c$ in place of  $l$.

 \subsection{\tt Some periodicities and the proof of Theorem \ref{T.dega}} \label{SS.per} Now,  for any monic divisor $r=r(t)$ of $D_1(t)$, of degree $r^*<d/2$,  consider the corresponding modulus $\tt m^*$ and, for an    integer $\lambda\in [1,d-2r^*]$     let us denote by $\cA(r, \lambda)$ the set of integers  $k\ge 0$ such that $k\delta\in W_{d^*-\lambda-r^*}(\tt m^*)$. 
 
%  If $\cA(r,\lambda)$ is finite, then \eqref{E.class*} with  $l=\lambda$ may hold only for bounded $\deg q$, and hence only finitely many convergents $(p_n,q_n)$ to $\sqrt{D^*}$ have the properties that $ \gcd(p_n,D_1)=r$ and that $\deg a_n=\lambda$. 
 
% Let us then assume that  $\cA(r,\lambda)$ is infinite. 

Taking into account that $W_{d^*-\lambda-r^*}(\tt m^*)$ is a constructible set, we may then apply  Corollary \ref{C.SML} to   this situation, on taking therein $\gamma:=\delta$, $\Gamma=G(\tt m^*)$.

%%%%%%%%%%%%%%%%%%%%%%%%%%%%%%%%%%%%%
\begin{comment}
We conclude  that  the Zariski closure   $Z(r,\lambda)$ of  
the set $\{k\delta: k\in \cA(r,\lambda)\}$  in $G(\tt m^*)$ consists, up to a finite set,  of a finite union of cosets of the connected algebraic subgroup $\Delta_0(\tt m^*)$ (defined in \S \ref{SSS.can}).  In particular, this subgroup   must have positive dimension. 

By Lemma \ref{L.SML}, we also conclude that there exists an integer $q$ such that $\Delta_0(\tt m^*)$ is   the Zariski closure of the set of multiples $k\delta$ with $k\equiv 0\pmod q$.   So  $Z(r,\lambda)$  consists of a nonempty  finite union of sets $\{ k\delta: k\equiv c\pmod q\}$. 

Note now that certainly $Z(r,\lambda)\subset W_{d^*-\lambda-r^*}(\tt m^*)$. The first thing that we would like to obtain correspondingly to multiples $k\delta\in Z(r,\lambda)$  is an equation of type \eqref{E.class*}; 
\end{comment}
%%%%%%%%%%%%%%%%%%%%%%%%%%%%%%%%%%

\medskip

We conclude that $\cA(r,\lambda)$ is, up to a finite set, a finite union of arithmetical progressions modulo $\mu=\mu(\tt m^*)$, where $\mu$ is such that  the connected algebraic subgroup $\Delta_0(\tt m^*)$ (defined in \S \ref{SSS.can})  is the Zariski closure of all  the multiples of $\mu\delta$ in $G(\tt m^*)$.

\begin{proof}[Proof of Theorem \ref{T.dega}]   Suppose that $k$ is the degree of $p_n$ in a convergent pair $(p_n,q_n)$ to $\sqrt D$, and that $l=\deg a_n$. Then, we have seen  in  \S \ref{SS.proof1} that  if $r=\gcd(p_n,D_1)$, then we may associate to the convergent the multiple $(k-r^*)\delta$ inside a set $W_{d^*-l-r^*}(\tt m^*)$. We also have seen that these multiples, for large $k$,  make up a certain finite union of    arithmetical progressions.

To prove the theorem, reciprocally, we shall  analyze the converse assertions.

\medskip

We proceed to  prove the theorem simultaneously for all divisors $D^*$ of $D$, and we do this by decreasing  induction on $\deg a_n$, which is anyway $\le d$.

\medskip

Since $\deg a_n\le d$, we may use as a starting point  for the induction the `empty' case $\deg a_n=d+1$: now there are no convergents and hence our assertions are true.

\medskip

{\tt Inductive assumption}: Suppose now that $1\le\lambda\le d$ and to have proved   that, for every $l>\lambda$,  the set of  integers $k$ such that there exists a convergent $(p_n,q_n)$ with $\deg p_n=k$ and partial quotient $a_n$ of degree $l$ 
is,  up  to a finite set,  a certain finite union (possibly empty) of arithmetical progressions modulo the least common multiple $\Pi$ of the possible $\mu(\tt m^*)$ which occur. Suppose we have proved this not merely for $D(t)$ but also for any divisor $D^*(t)$ of $D(t)$ such that $D/D^*$ is a square.

\medskip

We now prove that this holds also for $l=\lambda$. 

Since a divisor $D^{**}$ of   $D^*$  such that $D^*/D^{**}$ is  a square is also a divisor of $D$ with the same property, we may argue directly with the convergents to $\sqrt D$.

\medskip

Consider then a large integer $k$, where we are interested in whether $k=\deg p_n$ for a convergent $(p_n,q_n)$ to $\sqrt D$ with partial quotient $a_n$  of degree $\lambda$. We shall partition the set of possible $\deg p_n$ into subsets, in each of which the sought possibility depends only on a congruence modulo $\Pi$. 

\medskip

A first case occurs when both of  the following conditions hold: 

(i)  there are 
 a proper divisor $D^*$ as above, $D=r^2D^*$,   and a convergent $(a,b)$ to $\sqrt D^*$, with partial quotient  of degree $= \lambda+r^*$ and $\deg a=k-r^*$;
 
 (ii)  there is no divisor  $s$ of $r$ of positive degree and a convergent $(a',b')$  to $s\sqrt{D^*}$ with $\deg a'=\deg a=k-r^*$ and partial quotient of degree $\lambda+r^*+\deg s$. 
 
 \medskip
 
 Note that by the inductive assumption,  each of (i), (ii), and thus their union,  depends (for large $k$)  only on the classes of  $k$ relative to the various moduli $\mu(\tt m^*)$ which occur.
 
 We contend that for these values of $k$  there is a convergent $(p_n,q_n)$ to $\sqrt D$ with $\deg a_n=\lambda$ and $k=\deg p_n$, so $k$ is indeed in the sought set.
 
  In fact, by (i) we have that $\ord(ra-b\sqrt D)=\deg b+\lambda$ while $\deg (ra)=k$, so certainly $ra/b$ is a convergent to $\sqrt D$ and it suffices to prove that $r,b$ are coprime. Now, if $s=\gcd(r,b)$, then $(a,b/s)$ is  a convergent to $s\sqrt{D^*}$  with  partial quotient of degree $\lambda+r^*+\deg s$. If $\deg s>0$ this goes  against (ii), so indeed $\gcd(r,b)=1$. 
  
  Therefore we can detect the set of degrees $k$ of $p_n$ which fall into this situation, in the sense that  they form for large $k$ precisely a finite union of certain arithmetical progressions modulo $\Pi$.

\medskip

Supposing that $k$ is not in  such set,  assume  $k=\deg p_n$, for a convergent $(p_n,q_n)$ to $\sqrt D$ with partial quotient  of degree exactly $\lambda$. We proceed to prove that $p_n,D_1$ are coprime.

In fact, put  $r=\gcd(p_n,D_1)$.  If $r^*:=\deg r>1$, then $(p_n/r, q_n)$ is  a convergent to $\sqrt{D^*}$, for the proper divisor $D^*=D/r^2$ of $D$, with partial quotient of degree $\lambda+r^*$, and hence (i) is satisfied.

We contend that (ii) is also true. In fact, suppose by contradiction that  there is a  divisor  $s$ of $r$ of positive degree and a convergent $(a',b')$  to $s\sqrt{D^*}$ with $\deg a'=\deg a=k-r^*$ and partial quotient of degree $\lambda+r^*+\deg s$.  Then we would have 
$\ord(a'-b's\sqrt{D^*})=\deg b'+\lambda+r^*+\deg s$, whence $\ord(ra'-b's\sqrt D)=\deg b'+\lambda+\deg s$.  Note also that $\deg ra'=k,\deg b'+\deg s=k-d=\deg q_n$. 

But then $ra'q_n-b'sp_n=q_n(ra'-b's\sqrt D)-b's(p_n-q_n\sqrt D)$ would have order $\ge \min (\deg q_n+\lambda-\deg b'-\deg s, \deg b'+\lambda+\deg s-\deg q_n)=\lambda\ge 1$ at $\infty_+$. But this implies $ra'q_n=b'sp_n$, whence $p_n$ would divide $ra'/s$, which however has smaller degree and does not vanish; hence we have the sought contradiction.

\medskip

Since we have previously excluded the $k$ falling into both (i) and (ii), we conclude that for a possible partial quotient as above we would have indeed  $(p_n,D_1)=1$.

\medskip

As we have seen in \S \ref{SS.proof1}, letting $\varphi_n=p_n-q_n\sqrt D$, we have that 
 \begin{equation*}
 \div(\varphi_n)=-k\delta +(\sigma -(d-\lambda)(\infty_+)),
 \end{equation*}
 
where $\sigma$ is an effective divisor prime to $\tt m$ and of degree $d-\lambda$, sum of points $x_i$. 

In particular, since $\varphi_n/\varphi_n'$ is congruent to $1$ modulo $\tt m$, we derive that 
$k\delta\in W_{d-\lambda}(\tt m)$, hence $k\in \cA(1,\lambda)$, and from now on we can restrict  further to these values of $k$, for otherwise there is no convergent with the stated properties.

 Again, for large $k$ all of  these conditions leave us with finitely many arithmetical progressions modulo $\Pi$.

\medskip

Before performing a kind of converse deduction, we exclude still other values of $k$. Namely, let us consider the (large) integers $k$ such that  for some integer $h\in [1,d]$  there is a convergent $(a,b)$  to $\sqrt D$ with $\deg a=k-h$ and partial quotient of degree $\lambda+h$.

Note that  in view of the inductive assumption these values of $k$ too  form (for large $k$) precisely certain  arithmetical progressions modulo $\Pi$. 

We contend that if $k$ is in such last defined  set we cannot have a convergent $(p_n,q_n)$ with $k=\deg p_n$ and $\lambda=\deg a_n$. In fact, if this were the case, we would have $q_na-p_nb=q_n(a-b\sqrt D)-b(p_n-q_n\sqrt D)$. However this expression has an order at $\infty_+$ which is  $\ge\min( \deg b+\lambda+h-\deg q_n, \deg q_n+\lambda -\deg b)\ge \lambda\ge 1$. This would force  $q_na=p_nb$, which is a contradiction since $p_n$ cannot divide $a$.

\medskip
 
 Take now a large $k\in \cA(1,\lambda)$ which does not meet both (i) and (ii) above and  which does not lie in the set just considered; we also exclude the $k=\deg p_n$  for  convergents  $(p_n,q_n)$ to $\sqrt D$ with partial quotient of degree $>\lambda$: by induction, we may assume that these values as well form  precisely certain arithmetical progressions modulo $\Pi$.   
 
 Since $k$  lies in $\cA(1,\lambda)$, we have by defnition that  $k\delta\in W_{d-\lambda}(\tt m)$.

 As above, there is then a rational function $f$ on $\widetilde H$ of the shape $f=a(t)-b(t)\sqrt D$ with polynomials $a,b$ and
 \begin{equation*}
 \div(f)=-k\delta +(\sigma -(d-\lambda)(\infty_+)),
 \end{equation*}
 where $\sigma$ is an effective divisor prime to $\tt m$ and of degree $d-\lambda$, sum of points $x_i$. Since $\sigma$ is prime to $\tt m$, we have $\gcd(a,D_1)=1$. 
 
 As we have   seen above,  some cases may occur.
 
 \medskip

 The divisor relation implies that   $f$ has poles at most  at $\infty_\pm$, with pole orders $\le k$, and for large $k$ it has certainly a zero at $\infty_+$. Hence $\deg a\le k$, $\deg b=\deg a-d\ge k-2d$. 
 
 Then, certainly $a/b$ is a convergent to $\sqrt D$, because $\ord_{\infty_+}(f)\ge k-d+\lambda\ge \deg b+\lambda>\deg b$.

 \medskip

 To explore more precisely the orders of poles and zeros of $f$, let us think of the divisor $\sigma =\sum (x_i)$.
 
 Suppose that some $x_i$ equals $\infty_+$. Then we may omit it, replacing $\lambda$ with $\lambda+1$; so in fact $k\in \cA(1,\lambda+1)$ (and $\ord_{\infty_+}(f)\ge k-d+\lambda+1$). 
 
 We know that  $k\in \cA(1,\lambda+1)$ holds (for large $k$) precisely if $k$ lies in certain arithmetical progressions modulo $\mu(\tt m)$, which divides $\Pi$. 
 
 In this case certainly $k$ is not the degree of a $p_n$ with partial quotient of degree $\lambda$, because then  the function $\varphi_n=p_n-q_n\sqrt D$ would vanish at $\infty_+$ to order $\deg q_n+\lambda=k-d+\lambda$.  Then the polynomial $aq_n-bp_n=q_nf-b\varphi_n$ would have order at $\infty_+$ at least $\min (-\deg q_n+\ord(f),-\deg b+\ord(\varphi_n)\ge \min(d-k+k-d+\lambda+1, -k+d+k-d+\lambda)\ge 1$ and this implies  that $aq_n=bp_n$, whence $(a,b)$ would be a constant times $(p_n,q_n)$ and the partial quotient would have degree $\ge \lambda+1$. 
 
 Hence we may assume that no $x_i=\infty_+$. It follows that $\ord_{\infty_+}(f)=k-d+\lambda$. 
 
 Now, if $s=\gcd(a,b)$ has degree $h$, we have that $(a/s,b/s)$ is a convergent  %(which we have seen above) 
 whose partial quotient has degree $\ge h+\lambda$. But this fact has been taken  care of if $h>0$, in the sense that we have already noticed that the relevant values of $k$ cover certain arithmetical progressions, and we have excluded them. Therefore $a,b$ are coprime.
 
 \medskip

Suppose now that some $x_i$ equals $\infty_-$ (which implies itself $\lambda<d$), and let $h$ be the exact number of such points. Then the divisor relation would take the shape
\begin{equation*}
 \div(f)=-(k-h)\delta +(\sigma_1 -(d-\lambda-h)(\infty_+)),
 \end{equation*}
 where now $\sigma_1$ is an effective divisor of degree $d-\lambda-h$, with no $\infty_\pm$ among its points. 
 
 This also  implies that $\deg a= k-h$, $\deg b=k-h-d$. Again, we obtain that $(a,b)$  is a convergent falling into a previously excluded case.

 \medskip

We have established that no $x_i=\pm\infty$. This entails that $\deg a=k$ and that $\ord_{\infty_+}(f)=k-d+\lambda$. Since $a,b$ are coprime, we find that $k$ is a degree of the required shape.

\medskip

This takes into account all possibilities and proves the contention by induction.

The theorem as stated  in the Introduction is an immediate consequence: we have proved that  for large $n$ the degrees of the $p_n$ constitute precisely a certain set of arithmetical progressions modulo $\Pi$. So for large $m$ every interval $[m\Pi, (m+1)\Pi)$ contains the same number of $\deg p_n$, arranged in the same pattern. Recalling that the $\deg a_n$  are the differences of two consecutive ones among the $\deg p_n$, it follows that their sequence is indeed eventually periodic, of period (dividing the) number of $\deg p_n$ in the said interval.
\end{proof}

\subsubsection{\tt Remarks and examples} \label{SSS.rem.ex}   We collect here a number of  issues which we shall not develop in  detail here, in spite of their relevance.

\medskip

\noindent{\tt About the period}. Except for $\tilde g=0$ (see examples below), the given proof does not allow any good information on the anti-period, nor to establish the actual period length. This issue is related to effectivity in the Skolem-Mahler-Lech Theorem (especially in the case $\tilde g=0$ of Example \ref{EX.g=0} below), which is not yet known (and even considered possibly undecidable by some authors).

 In concrete cases however (as observed in \S \ref{S.SML}) the proof allows to bound  the period length effectively from above. We have no idea on the variation of the length with the data; for instance, one could ask whether the period length  may be bounded in terms only of $d$. These issues appear to be very deep. 

\medskip

\noindent{\tt  Subvarieties of Jacobians}. We again remark that `often' all the $\deg a_n$ shall be eventually $1$. However  it may happen that all of them are larger: just substitute $t\mapsto  t^h$ throughout.

 Let us comment on this with a bit more detail, restricting for the moment to the squarefree case (i.e. $D_1=1$), which is rather less complicated. Let then $(p,q)$ be a convergent to $\sqrt D$ with partial quotient of degree $l$. 

In this case equation \eqref{E.J} produces a multiple $(\deg p)\cdot \delta$ inside the subvariety $W_{g-(l-1)}$ of the Jacobian $J$.\footnote{Recall this is closed in the case of the usual Jacobian, whereas it is only constructible in general.}  Now, if $l>1$ this is a proper subvariety, and if this happens for infinitely many multiples, the results of \S \ref{S.SML} imply that  a translate (actually by a torsion point) of the canonical abelian (sub)variety $\Delta_0$ is contained in $W_{g-l+1}$.  

The abelian subvarieties of the $W_m$ have been studied in connection with points of bounded degree on curves, and we refer to  \cite{DF}, \cite{F} for more.  Clearly, if for instance $J$ is simple, either $\delta$ is torsion (and $D(t)$ is Pellian) or the above implies $\Delta_0=J$. In turn, this yields that $l=1$ for all but finitely many convergents. Again, even in this case  I do not know of any method for establishing  effectively the last occurrence of degree $>1$ (except for  $g\le 1$).  Similar considerations hold for the non-squarefree case, with generalized Jacobians in place of $J$. 

Here are some explicit examples in low genus (see \cite{Ber} and \cite{BMPZ} for other ones).

\begin{small}
\begin{example}\label{EX.g=0} Let us start with $\tilde g=0$, and $D(t)=D_1(t)^2(t^2-1)$, assuming  for simplicity that $D_1$ has  $g=d-1$ simple roots $\rho\neq \pm1$. Now   $g=\deg D_1$, and $\tt m$ is the sum of $2g$ points $\xi_{\rho}^\pm=(\rho,\pm\sqrt{\rho^2-1})$ above  the $g$ roots   of $D_1$. %We can work into the $G_\rho$ corresponding to each root. 
A divisor $A$ of degree zero is always principal $=\div(f)$, and we have a homomorphism to $\G_{\rm m}$ given by $A\mapsto f(\xi_\rho^+)/f(\xi_\rho^-)$; assembling these $g$ homomorphisms we obtain the isomorphism  $G\cong \G_{\rm m}^g$ for  the  generalized Jacobian. %By Weil's reciprocity (see \cite{SeAGCF}), and since $\xi_\rho^+-\xi_\rho^-=\div ( we have 
The divisor $\infty_--\infty_+$ equals $\div(z)$, where $z=t+u$, so $\delta\mapsto z(\xi_{\rho}^+)/z(\xi_{\rho}^-)=z(\xi_{\rho}^+)^2$. A point $p%=(\alpha,\beta)
\in\widetilde H$  corresponds  to $\div(z-z(p))$ and is sent to %$(\rho-\alpha-\beta+\sqrt{\rho^2-1})/(\rho-\alpha-\beta-\sqrt{\rho^2-1})=
$(z(p)-z(\xi_\rho^+))/(z(p)-z(\xi_\rho^-))$. The varieties $\overline{W_h}=\overline{W_h(\tt m)}\subset \G_{\rm m}^g$ are then described by explicit equations; it is a pleasant   exercise to show that  $\overline{W_{g-1}}$  contains no coset of an algebraic subgroup of positive dimension.\footnote{This amounts to say that if $z_1,\ldots ,z_{g-1}$ are not all constant  functions on a curve, then the functions $\prod_i((z_i-a_j)/(z_i-b_j))$, $j=1,\ldots ,g$ generate,  modulo constants,  a  multiplicative subgroup of rank $>1$, for $a_j,b_l$ pairwise distinct; one looks at  zeros/poles.}  Then a coset of  the algebraic subgroup $\Delta_0(\tt m)$  can be contained in it only if $\Delta_0(\tt m)$ is trivial, i.e. all values $z(\xi_{\rho}^+)$ are roots of unity (i.e.  $D(t)$ is Pellian).\footnote{This torsion case is discussed also in \cite{McM2}, using Chebyshev polynomials.}   In any case, the partial quotients of degree $> h$ correspond to powers of the image of $\delta$ contained in $W_{g-h}$.  In many `concrete' cases, for $h>0$  one may find all such (finitely many) values, but I do not know of any completely general such procedure which is effective (except when $W_m$ is a curve).  Anyway,  this discussion proves  that: 

\smallskip

{\it Either $D(t)$ is Pellian, which happens if and only if all $z(\xi_{\rho}^+)$ are roots of unity, or there are only finitely many partial quotients of degree $>1$.} 
\end{example}

\smallskip   

The literature apparently contains only  the Pellian case (recalled  e.g. by  McMullen in  \cite{McM2}). 

\begin{example}\label{EX.g=1}  Let $D(t)=t^4+t^2+t$, which yields an elliptic curve $H$ with origin $\infty_+$; standard methods confirm that  $\delta$ is non-torsion, hence $D$ is non-Pellian. Now all partial quotients except $a_0$ have degree $1$. 

(ii) If we modify to $D(t)=t^2(t^4+t^2+t)$, the relevant generalized Jacobian $G$  is an extension of 
$H$ by $\G_a$, and it is non-split (see \cite{SeAGCF} and \cite{CMZ}). It follows that $\Delta_0$ is the full $G$, so again all partial quotients shall be eventually $1$. Incidentally, this also proves that only finitely many denominators of the convergents to $\sqrt{t^4+t^2+t}$ vanish at $0$ (for if $q_n(t)=t\hat q(t)$ we have a convergent $(p_n,\hat q)$ to $\sqrt{D}$ with partial quotient of degree $2$).

Recall  also that any partial quotient of degree $2$ yields a multiple $k\delta$ inside $W_1(\tt m)$; we do not know of any general method to calculate all such multiples, though an analogue of the proof method of \S \ref{S.SML} could sometimes work. 

(iii) If we modify to $D(t)=(t-\rho)^2(t^4+t^2+t)$, $\rho$ nonzero and not a root of $t^4+t^2+t$, $G$ is an extension of $H$ by $\G_{\rm m}$, isogenous to a split one precisely if $\xi_\rho^+-\xi_\rho^-$ is  torsion on $J$ (where $\xi_\rho^\pm$ are the points of $H$ above $t=\rho$). If this is not the case, we have similar conclusions as before. If it is, the situation depends on whether $\dim\Delta_0=1,2$. The   last case is similar to the above. To check whether the dimension is $1$, we may argue as follows. Let $\psi$ be a function on $H$ with divisor $m(\xi_\rho^+-\xi_\rho^-)$. Then, if $A=\sum m_i (x_i)$ is a divisor of degree $0$ on $H$,  %prime to $\tt m$ 
we have a map   $A\mapsto \prod\psi(x_i)^{m_i}$, and this maps $G$ to $\G_{\rm m}$.  Then an algebraic subgroup of $G$ different from $\G_{\rm m}$ is the kernel of this map, and hence  it  follows  that  $\dim \Delta_0=1$  only if $\Delta_0$  is inside this kernel, i.e.  $\psi(\infty_-)/\psi(\infty_+)$ is a root of unity. Precisely  in this case we have an infinity of partial quotients of degree $2$. 

However we may show this happens at most finitely many times (and perhaps never  for this $H$,  a fact which possibly one can prove). Indeed, if we have $\psi^k(\infty_-)=\psi^k(\infty_+)$, then the divisor $\xi_\rho^+-\xi_\rho^-$  would be torsion, this time in the extension $\cG$ of $H$ by $\G_{\rm m}$ defined by the modulus $\infty_-+\infty_+$. % (one has to choose another origin for $H$). 
This extension is not isogenous to a split extension, because $\delta$ is not torsion in $H$. But the set of divisors classes $x-x'$,  $x\in H$,  forms a curve in $\cG$, which is not an algebraic subgroup (for instance since $\cG$ is not isogenous to a split extension). Then a theorem of Hindry \cite{H} applies. 

We have paused so long on  this  example also because the last conclusion is related (as in (ii)) to another result of this paper; namely, it implies that: {\it There are only finitely many numbers which are roots of infinitely many denominators $q_n$ of the convergents to $\sqrt{t^4+t^2+t}$.} Indeed, let $\rho$ be a root of such a $q_n$, so $q_n(t)=(t-\rho)b(t)$. Then $(p_n(t),b(t))$ is a convergent to $(t-\rho)\sqrt{t^4+t+1}$, and the partial quotient has degree at least $2$, concluding the argument.\footnote{It is to be remarked that several of these conclusions would follow also from Theorem \ref{T.dega2}; however we think these indepedent arguments may be relevant for other purposes.}  See also Remark \ref{R.ber} below. 

\end{example}

\begin{example}\label{EX.g=2} (i) Let now $D(t)=t^6+t+1$. It may be checked that this has genus $2$, that $J$ is simple   and again, since $\delta$ may be checked to be non-torsion,  all partial quotients have eventually degree $1$. This follows independently of the simplicity of $J$, because otherwise $W_1\cong H$ would have to be an elliptic curve. (See \cite{MZ} and the Appendix by V. Flynn for a discussion of the $\lambda$  when {\it some} partial quotient relative to $t^6+t+\lambda$ has degree $2$ and a proof of finiteness of the $\lambda$ for which the degree may be $3$, i.e. the Pellian cases in the family.)

We can also repeat some considerations of the previous example, on  modifying  to $D(t)=t^2(t^6+t+1)$or  to $D(t)=(t-\rho)^2(t^6+t+1)$

The  arguments apply more  generally; also, on varying the polynomial  in the family $t^6+at^4+bt^3+ct+1$, dimensional considerations suggest that we should  find  (non-Pellian) cases in which indeed infinitely many of the $q_n$ have a common zero. % (in fact, the last deductions may be reversed).   
\end{example}
%\begin{example}\label{EX.g>2}  

Similar  examples of course are possible in higher genus. % If  for instance $D(t)$ is squarefree of degree $8$, non Pellian, and the Jacobian is simple (as can be often checked), all but finitely many of the partial quotients  have   degree $1$: for otherwise the variety denoted above $W_2$ would contain an elliptic curve, and the Jacobian would not be simple.  \end{example}
\end{small}

\subsection{\tt Proof of Theorem \ref{T.dega2}}  Let us assume that there are infinitely many convergents $(p,q)$ to $\sqrt D$ with partial quotient of degree $l>d/2$. As in \S \ref{SS.proof1}, let us put $r:=\gcd(p,D_1)$, $D=r^2D^*$, $p=rp^*$, $r^*=\deg r$.  We may pick an $r$ of maximal degree  which occurs infinitely many times. As in \S\ref{SS.proof1}, setting $\varphi^*:=p^*-q\sqrt{D^*}$, we have 
 \begin{equation*}
 \div(\varphi^*)=-(\deg q+d^*)\delta +(\sigma^* -(d^*-l-r^*)(\infty_+)),
 \end{equation*}
where $\sigma^*=\sum (x_i)$ is an effective divisor of degree $d^*-l-r^*$ on $\widetilde H$ prime to $\tt m^*$, with no $x_i=\infty_\pm$, and with no pair $x_i,x_j$, $i\neq j$,  conjugate under our involution.

As already observed,  $(p^*,q)$ is a convergent  to $\sqrt{D^*}$ with partial quotient of degree $l+r^*$. By Theorem \ref{T.dega}, applied to $D^*$ in place of $D$, since this holds for an infinity of convergents, there is a whole arithmetical  progression of $k$  for which this holds for all large integers in it with $\deg p^*=k$.\footnote{This indeed follows from Theorem \ref{T.dega}, but anyway has been explicitly shown during the proof.}  Let $\{m\Pi+c, m\in\N\}$, be such a progression.  

Now, {\it a priori} it could happen that $p^*,D_1^*$ are not coprime along the whole progression;\footnote{Actually, one could strengthen Theorem \ref{T.dega} to include this, but we shall not need it.}  however this can happen at most finitely many times, because otherwise $r^*$ would not be maximal, as we have assumed before.  Hence we may assume that the last displayed equation holds for all elements in our progression, with $m\Pi+c$ in place of $\deg q+d^*$.

Let us now denote  by $\sigma^*_m$  the divisor $\sigma^*$ corresponding to the integer $m\Pi+c$ in the progression. Summing the equations corresponding to $m-1,m+1$ and subtracting twice the one corresponding to $m$, we get
\begin{equation*}
\sigma^*_{m-1}+\sigma^*_{m+1}\approx 2\sigma^*_{m},
\end{equation*}
where the strong equivalence is the one relative to $G(\tt m^*)$, as explained in \S \ref{SSS.gj}.  Now,  by definition of strong equivalence,  $\sigma^*_{m-1}+\sigma^*_{m+1}- 2\sigma^*_{m}$ is the divisor of a function $f_m$  to which we can apply Lemma \ref{L.unique}; since this function has degree at most $2(d^*-l-r^*)\le 2d-2l-2r^*<d^*$, the conclusion of the lemma implies that $f_m\in \C(t)$.\footnote{A somewhat related  argument appears in Frey's paper \cite{F} on points of bounded degree.} So, if a point $(\xi)$ appears in the divisor of its poles, also $(\xi')$ must appear, and by the above  this is only possible if $\xi$ corresponds to a zero of $\widetilde D$ distinct from the zeros of $D_1$. No other poles are possible. 

Also, the multiplicity of $(\xi)$ in  $\sigma^*_m$ must be exactly $1$, for otherwise the corresponding $p^*,q$ would not be coprime. Moreover $(\xi)$ cannot appear in $\sigma^*_{m+1}$, for otherwise it would appear only with multiplicity $1$ as a pole of $f_m$, which could not be a divisor of a function in $\C(t)$.

Now, if $(\xi)$ indeed appears in $\sigma^*_m$,  thus with multiplicity $1$, it  would appear in the divisor of zeros of $f_{m+1}$, and hence would  appear there with multiplicity $\ge 2$; this implies that it would  also appear in $\sigma^*_{m+2}$ (again with multiplicity $1$).   

\medskip

Similar considerations hold for the zeros. Let us suppose that a $\xi$  not of the said type, hence $\xi'\neq \xi$,  appears among the zeros of $f_m$. Then  $\xi'$ has also to appear since $f_m\in\C(t)$. But  we know that $\xi,\xi'$ cannot appear simultaneously in a same $\sigma_n^*$, hence by symmetry we may assume that $\xi$ appears  in $\sigma_{m+1}^*$, and $\xi'$   in $\sigma_{m-1}^*$; then $\xi$ cannot appear in $\sigma_{m-1}^*$ and therefore its multiplicity $\mu$ in $\sigma_{m+1}^*$ is greater that  twice the multiplicity $\nu$ in $\sigma_m^*$.  Now look at $f_{m+1}$; since $\xi$ cannot be among its poles (by the above), its multiplicity in $\sigma_{m+2}^*$ has to be at least $2\mu-\nu>\mu$. Now, repeating the last argument with $f_{m+2},f_{m+3},\ldots$, we see that $\xi$ would have strictly increasing multiplicity in the subsequent $\sigma_n^*$, which eventually is impossible.

\medskip

In conclusion,  every $\xi$ which appears has to be of the above type, and for every such $\xi$ the pattern of appearance is    periodic of period $2$, hence $\div(f_m)=\div(f_{m+2})$, whence
\begin{equation*}
\sigma^*_{m-1}+\sigma^*_{m+1}- 2\sigma^*_{m}=\sigma^*_{m+1}+\sigma^*_{m+3}- 2\sigma^*_{m+2}.
\end{equation*}
%The same recurrence must then hold for the multiplicities in the $\sigma^*_m$ of any given point. Only finitely many points can occur, because if one of them does not occur in any of five consecutive terms it will never occur, and because the $\sigma^*_m$ are effective of bounded degree. 
This linear recurrence has `roots' $1$ and $-1$. It is not difficult to check that any solution in a free abelian group  is of the shape $\alpha+(-1)^n \beta+n\gamma$, with $\alpha,\beta,\gamma$ in the group, in this case the divisors. Since our solution consists of effective divisors of  bounded degree, we must have $\gamma=0$ and in particular, we must eventually have $\sigma^*_{2m}$ constant.

But then, subtracting two of the divisor equations in the opening arguments, corresponding to consecutive multiples of $2\Pi$, we find that $2\Pi\delta\approx 0$ with respect to $G(\tt m^*)$, which means that  $D^*$ is Pellian.

Now we have only to check the stated inequality on $d^*$. Since $\sigma^*$ is effective we have $d^*\ge l+r^*=l+(d-d^*)>(d/2)+d-d^*$, i.e. $2d^*>3d/2$, as asserted.

\subsection{\tt Proof of Theorem \ref{T.h}}  %We start with the (more involved) lower bound. 
We shall estimate the height by means of rational functions, and we shall need sufficiently many of them which are independent. For this task, we first prove the following lemmas:

\begin{lem} Let $A$ be a simple (complex) abelian variety of dimension $r$, let $\alpha \in A$ a point such that the multiples $\{m\alpha: m\in\N\}$ are Zariski-dense in $A$, and let $f$ be a non-constant rational function on $A$. Then the rational functions $f(x), f(x+\alpha),\ldots ,f(x+(r-1)\alpha)$ are algebraically independent on $A$.
\end{lem} 

\begin{proof}  Let  $\partial_1,\ldots ,\partial_r$ be independent derivations on $A$, invariant by translation. Then if the conclusion is not true, the gradient vectors $F_s=F_s(x):=(\partial_1 f(x+s\alpha),\ldots ,\partial_r(f(x+s\alpha))$, $s=0,\ldots , r-1$, are linearly dependent over the function field $\C(A)$ of $A$. Let $m\ge 0$ be the maximal integer such that $F_0,\ldots ,F_{m-1}$ are linearly independent (so $m=0$ iff $F_0=0$). Then if $m=0$ we have that $f$ is constant, against the assumption; hence $m\ge 1$, and $m\le r-1$ under the present  hypotheses. 

Note that, since the $\partial_i$ are translation-invariant,  replacing $x$ by $x+h\alpha$ (any $h\in\Z$)  shows that  $m$ is the maximal integer such that any $m$ consecutive ones among the $F_s$, $s\in\Z$, are independent, and any consecutive $m+1$   of them are  dependent. Then by induction on $s\in\N$ it is easy to see that $F_0,\ldots , F_{m-1}, F_s$ are dependent (use that $F_s$ lies in the space spanned by the $m$ preceding vectors). Hence for any $s$ the vector $F_0(x+s\alpha)=F_s(x)$ lies in the $\C(A)$-space generated by $F_0,\ldots , F_{m-1}$, which means that all  $(m+1)\times (m+1)$ minors of the corresponding matrix vanish, as rational functions on $A$. By Laplace rule, any minor is of the shape $\sum_{i\in I}c_i(x)\partial_if(x+s\alpha)$, where $I$ is a subset of $\{1,\ldots ,r\}$ with $|I|=m$ and where the $c_i$ are rational functions on $A$, independent of $s$, and,  at least for some $I$,  not all zero.  Now, since the vanishing holds for all $s$ and since the multiples of $\alpha$ are Zariski-dense, we have the same relation on replacing $s\alpha$ by any point $z\in A$. Hence the vector $(\partial_1f)(x+z),\ldots ,(\partial_rf)(x+z))$ satisfies a nontrivial  linear relation with coefficients which are rational functions only of $x$. By specializing $x$, we obtain that there is a non-zero derivation $\partial$ invariant by translation and such that $\partial f=0$ identically. Hence $f$ is constant on a non-trivial subtorus (of the complex torus corresponding to $A$); the Zariski closure   in $A$ of this  subtorus is an  abelian subvariety  and then since $A$ is simple $f$ must be constant, a contradiction that proves what asserted. 
\end{proof} 

\begin{lem}\label{L.h}   Let $\Delta$ be an abelian variety over $\overline\Q$, let  $\delta\in \Delta(\overline\Q)$ be such that its multiples are Zariski dense in $\Delta$, and let $f\in\overline\Q(\Delta)$ be  non-constant. Then there is an integer $m>0$ such that for any  integer $n>0$  at least one of the functions $f(x+h\delta)$, $0\le h\le m$, is defined at $x=n\delta$ and 
$$
 1+\max_{h=0}^m h(f(n+h)\delta)) \gg n^2,
$$ 
where   in taking the maximum we consider only  the (non-empty) set of well-defined values, and where the implicit constant does not depend on $n$.
\end{lem}

\begin{proof} The assertion  is invariant under isogeny, so we may suppose that $\Delta$ is a product of simple abelian varieties (defined over $\overline\Q$); then $f$ is non-constant when restricted to some simple factor $A$ of   dimension $r>0$. We let $\alpha$ be the projection of $\delta$ to $A$, so  the multiples of $\alpha$ are Zariski-dense in $A$. Note that if $\hat h$ is a canonical height on $A$ associated to an ample divisor, we have $\hat h(n\alpha)=n^2\hat h(\alpha)\gg n^2$, because $\alpha$ is not a torsion point of $A$. 

By the previous lemma, the functions $f(x),\ldots ,f(x+(r-1)\alpha)$ on $A$ are algebraically independent, hence we obtain a dominant rational map 
$$
F:A\to\A^r, \qquad F(x)=(f(x),\ldots ,f(x+(r-1)\alpha)).
$$
Let $V$ be a closed proper subset of $A$ such that $F$ is defined  on $A\setminus V$. On   enlarging  $V$ if necessary, we may assume that $F$ is finite on $A\setminus V$ to its image, and then   standard (easy) arguments on heights show that for algebraic points $x\in A\setminus V$, we have $1+h(F(x)) \gg \hat h(x)$. Then,  for any fixed integer $j$, if $x\in A\setminus (V-j\alpha)$ we have $1+h(F(x+j\alpha))\gg \hat h(x+j\alpha)\gg \hat h(x)+O(1)$.

Now, since the set of multiples of $\alpha$ is Zariski-dense, there is an integer $b>0$ such that the intersection $\bigcap_{j=0}^b (V-j\alpha)$ is empty.\footnote{It is not difficult to show that $b$ can be bounded only in terms of the number, dimensions and degrees of the components of $V$.} Then for every $x_0\in A(\overline\Q)$ at least one among the $F(x+j\alpha)$, $0\le j\le b$, is defined at $x=x_0$, and by the above we have 
%Then, for every $x\in A$ we have 
$1+\max_{j=0}^b h(F(x_0+j\alpha))\gg \hat h(x_0)+O(1)$, and the conclusion of the lemma follows on taking $m=r+b$, $x_0=n\alpha$. 
\end{proof}

\begin{small}
\begin{rem}\label{R.vojta}  It would be desirable to have a lower bound for each individual value $h(f(n\delta))$; however this issue  appears to lie beyond the presently known techinques (except when $\dim\Delta =1$). The functorial properties of the height suffice when we  deal with values of morphisms; otherwise  deep problems arise already in simple cases, due to the appearance of exceptional divisors when we regularize the map by blowing-up. Use of the Vojta conjectures (see \cite{[BG]}, Ch. 14) should often take care of these issues. 
\end{rem}
\end{small} 

To go ahead, we shall use the formulae of \S\ref{SS.formulae}, and also   obtain  further ones. We stick to that notation, working with a convergent $(p_n,q_n)$ and often omitting the index $n$ for simplicity.  

Also, in view of Theorem \ref{T.dega}  we may tacitly move $n$ in an arithmetic progression   modulo a certain fixed integer $\Pi>0$ such that the degrees of $a_n$ depend only on the class of $n$ modulo $\Pi$, and the degree of $p_n$ is expressed by a certain fixed linear polynomial in $n$. We may also assume that $\Pi$ is such that the multiples of $\Pi\delta$ lie in the canonical abelian variety $\Delta_0$. %where the degrees of $a_{n-1},a_n,a_{n+1}$ are fixed. Let us denote them by $l_-,l,l_+$ respectively. Within this progression, we shall have $\deg p_n=$ a certain fixed  linear polynomial in $n$.

\medskip

In the Jacobian $J$ we have   formula \eqref{E.J}, i.e. $ (\deg p) \delta=j(x_1)+\ldots +j(x_{d-l})$ for $l=l_n=\deg a_n$ and points $x_i=x_{in}\in H$, distinct from $\infty_\pm$ and such that no pair of conjugate points appear (under the usual involution) and uniquely determined by $\deg p$. 

Setting, $t_i=t(x_i)$, we also have the polynomial $R(t)=R_n(t)=c_n\prod_{i=1}^{d-l}(t-t_i)\neq 0$,  and  we may write  $R_n(t)^{-1}=c_n^{-1}t^{l-d}(1+\rho_1t^{-1}+\ldots )$, where $\rho_h$ is a  certain universal symmetric function homogeneous of degree $h$ of the $t_i$. 

Hence, by \eqref{E.a_n}, i.e. $a_n=2(-1)^n\sqrt D/R_n(1+O(t^{-l-1}))$,  the coefficient of $t^{l-j}$ in $a_n$, for $j=1,\ldots ,l$, is a certain linear combination, with coefficients which depend only on $D(t)$, of the $\rho_h$, $h\le j$.  It is to be remarked that $\rho_j$ appears in such $j$-th coefficient. 

\medskip

After these preliminaries we can go to the actual proof. 
Note that when we express a point $z\in J$ as a sum $z=\sum_{i=1}^g j(u_i)$  the $u_i\in H$ are generically uniquely determined by the point $z$ on $J$. Then the function $\sum t(u_i)$ (which is a rational function on the $g$-th symmetric power of $H$)  may be viewed as  a rational function of  $z\in J$, and the same holds for any given symmetric polynomial in the $t(u_i)$. 

Taking this into account, we see that the coefficients of   $c_na_n$ are given by the values at the point $z_n=(\deg p_n)\delta$ of certain fixed rational functions on $J$. \footnote{These rational functions may in fact depend on the degree of $a_n$, but here $n$ varies along a progression where $\deg a_n$ is fixed and $\deg p_n$ is a certain fixed linear polynomial in $n$.}  

This implies (by standard easy height theory) that the projective height $h(a_n)\ll n^2$, proving the first assertion of the theorem.

The lower bound (for the affine height) we found  more laborious. The point $z_n$ will lie on a suitable (torsion) coset of the non-trivial abelian subvariety $\Delta_0$ of $J$ (introduced in \ref{SSS.can}), and this coset shall be fixed for the progression of $n$ in question. Then we may actually view these rational functions  as rational functions on $\Delta_0$

\medskip

 Now, if any of these functions is non-constant on $\Delta_0$, we may apply Lemma \ref{L.h}  (with $\Pi\delta$ in place of $\delta$) and deduce that the maximum  (projective) height of a sufficient number of consecutive $a_n$ (for $n$ in the progression) shall be bounded below by $\gg n^2$, as required.

Thus we may assume from now on that these rational functions are constant on $\Delta_0$, and it follows that $c_na_n$ is a  polynomial  independent of $n$ for the values of $n$ in question. We may actually assume that this holds for all the progressions modulo $\Pi$.   In this case, which we do not know  if  at all possible \footnote{ We note that  if we restrict to a single progression, this may happen; these cases correspond to a quite peculiar Jacobian, see Remark \ref{R.12} above. Excluding that this may happen even for {\it all}  the progressions would lead to a lower bound for the usual projective height instead of the affine height.},    we have to take advantge of $c_n$. 

\medskip

For this, let us consider the  polynomials $a_nR_n$ (with leading coefficient $\pm 2$). Note that (for $n$ in a fixed progression modulo $\Pi$ and under the present assumptions)  this  is the product of a constant polynomial (i.e.  $c_na_n$) times the polynomial   $\prod (t-t_i)$; hence its coefficients in particular do not involve $c_n$ in this expression, but only symmetric functions of the $t_i$, of degree $\le d$. As before,  these functions can be viewed as rational functions on $\Delta_0$, evaluated at a suitable multiple of $\delta$. 

\medskip

Suppose first that  all the $a_nR_n$ are constant in each progression modulo $\Pi$ (for large $n$). Then the $t_i=t_{in}$ may have only finitely many values for varying $n$, and hence the same holds for the $x_i$. Taking then two distinct $n,m$ which correspond to the same $x_i$ we deduce that  $(\deg p_n-\deg p_{m})\delta=0$, against the assumption that $\delta$ is non-torsion.

\medskip

Hence let us suppose that  for some progression, some $a_nR_n$ is not constant. Writing $\prod(t-t_i)=t^{d-l}+\sigma_1t^{d-l-1}+\ldots$, let us suppose that $\mu$ is the minimum integer such that, in some progression, $\sigma_\mu$  does not correspond to a constant rational function. 

Then, exactly as before, if any of these functions is non-constant, we can derive a lower bound $\gg n^2$ for the maximum height of a sufficient number of consecutive ones among the coefficients of $t^{d-\mu}$ in the polynomials $a_nR_n$.

However,  equation \eqref{E.a_n} says that $a_nR_n=S_n-S_{n-1}$ and a lower bound would follow similarly for the maximum height of the coefficients of $t^{d-\mu}$ in sufficiently many consecutive ones among  the $S_n$. 

But then, referring again to \S \ref{SS.formulae}, we may use the equation  $S_n^2=D+R_nR_{n+1}$.   

We are assuming that all the first $\mu$ coefficients in $R_n/c_n$ and $R_{n+1}/c_{n+1}$ are constant  (in any  progression of $n$ modulo $\Pi$). Then the same would hold for the first $\mu$ coefficients in their product $(R_n/c_n)(R_{n+1}/c_{n+1})$.  

Since $\deg R_m=d-l_m\le d-1$, we deduce that the height of the first $\mu+1$ coefficients in $S_n^2$ is bounded by $h(c_nc_{n+1})+O(1)$, and expansion of the square root shows that the same holds for $S_n$. 

%Then  we may expand $$S_n=\pm \sqrt D(1+{R_nR_{n+1}\over D})^{1/2}=\pm \sqrt D(1+c_nc_{n+1}(\gamma_0t^{-l-l'}+\ldots +\gamma_{\mu-1}t^{-l-l'-\mu+1}+O(t^{-l-l'-\mu})),$$ where $l=l_n,l'=l_{n+1}$ and $\gamma_0,\ldots ,\gamma_{\mu-1}$ are constant. 

But  now  we have a contradiction unless the height of $c_nc_{n+1}$ is $\gg n^2$ for some $n$ in any sufficiently large interval. Then $h(c_n)+h(c_{n+1})\gg n^2$ for   these $n$.

We deduce that the maximum height of  a large enough number of consecutive $c_n$ is bounded below as required. And this would finally prove that the affine height of one at least of the corresponding $a_n$ is likewise bounded below.

This concludes the proof.

\begin{small}
\begin{rem} \label{R.h}  (i)  The proof shows that  the integer $M$ can be taken $\le c\Pi$, where $c$ depends only on $d$ and $\Pi$ is a period for the sequence of degrees of the $a_n$. Perhaps this dependence can be eliminated. 
For instance, when for instance $J$ is simple the proof may be shortened, and the recourse to Theorem \ref{T.dega} may be avoided. In this case, as follows from  the considerations above  in this paper, we have eventually  $\deg a_n=1$,  hence the period is $1$ and one may get a lower bound for the height on taking a number of consecutive $a_n$ bounded in terms only of $d$.

(ii)  {\tt Upper bounds}.  Whereas a bound $\ll n^2$ for the (usual) {\it projective} height of the $p_n,q_n,a_n$ follows from  e.g. Siegel's lemma (as in \cite{Bo-Co}) or by noting that the zeros of $\varphi_n=p_n-q_n\sqrt D$ satisfy that bound, % it is not clear to us whether
 the same sort of upper bound does not generally hold for the {\it affine}  height of the same quantities, as shall be shown in the next Example \ref{EX.ell}. Heuristically, to explain  such a somewhat striking  behaviour,   we note that the $p_n,q_n$ which arise from the continued fraction are not normalized as one could perhaps expect, e.g. with coprime integer coefficients (when everything is defined over $\Q$): indeed, reduction modulo a prime $\ell$ produces always a Pellian polynomial, as must be the case over a finite field, and this may be seen to force $p_n,q_n$ to be divisible by $\ell$ for some $n$. (See \cite{[vdPT]}, \cite{Ma} and \cite{Me}.) 

The arguments used in the above proof suggest that the   affine height might depend on  quantities like $h(\sum_{m=0}^n(-1)^m f(m\delta))$, where $f$ is a rational function on $J$; a precise estimation of  this seems to fall outside the  standard  theory, although one can obtain an upper bund $\ll n^3$. % Already the case of elliptic curves seems to lead to  problems in the estimation of this quantity, on which  we do not pause further  in this paper (except for next example). %Take for instance a Weierstrass model, and $f=x$. The height has a contribution coming from finite places and another one coming from archimedean places. The former can be bounded again by $\ll n^2$, since the valuations are ultrametric. As to the latter, this is no more true, and 
%To study what happens leads to subtle issues,  %on working on the underlying torus, we are led to estimate sums of the shape  $\sum_{m=1}^n||m\lambda ||^{-2}$, where now $\lambda$  is an elliptic logarithm of $\delta$ and $||.||$ denotes the distance of a complex number to the lattice. 

(iii) In the case $d=2$ of elliptic curves everything becomes more explicit, and it is readily proved that   $h(a_n)\asymp  n^2$ holds for each individual large $n$.  In the next example we add precision to this and prove the assertion of the above Addendum.
\end{rem}
\end{small}

\begin{small}
\begin{example}\label{EX.ell} We let $D(t)=t^4+t^2+t$, which can be checked to be non-Pellian, so all of the $a_n$ after the first have degree $1$. As in the above proof, we can relate them to the values of  the function $t\in\Q(H)$ at multiples of $\delta$. Now $g=1$, which makes things rather simpler: one finds that, setting $z_n=t((n+1)\delta)$, we have \footnote{Some of the formulae appear also in \cite{A-R} with different notation.} 
$$
R_n=c_n(t-z_n),\qquad a_n=2(-1)^nc_n^{-1}(t+z_n),\qquad S_n=(-1)^n(a_0+\gamma_n),
$$
for constants $\gamma_n$ with $\gamma_0=0$. From the identity $S_n^2=D+R_nR_{n+1}$  one derives $2\gamma_n=(z_n+z_{n+1})^{-1}=c_nc_{n+1}$, $8\gamma_nz_nz_{n+1}=(2\gamma_n+1)^2$ for $n>0$. 

Since the $z_n$ are values at $n\delta$ of the function $z\mapsto t(z+\delta)$ on $H$, of degree $2$, this yields the sought information for the heights of the $z_n$ and $\gamma_n$ (and also $c_nc_{n+1}$), i.e. the heights are asymptotic to constant times $n^2$. The same holds for the {\it projective}  height of $a_n$.

The formulae also deliver striking identities like $2\gamma_n=-1-4(z_n^2-z_{n-1}^2+\ldots +(-1)^nz_0^2)$, obtained from $S_n=\sum_{m=0}^na_mR_m$. 

%However,  the behaviour  of the height of the individual $c_n$ is not clear to me. 

For the affine height, one needs information on the $c_n$. From the above, we find $c_{n+1}/c_{n-1}=(z_{n-1}+z_n)/(z_n+z_{n+1})$, whence 
  $c_n=\prod_{m=1}^{n/2}((z_{n-2m}+z_{n-2m+1})(z_{n-2m+1}+z_{n-2m+2})^{-1})$ (for even $n$). How does the height of this product behave ?  We analyze this in the 
  
  \medskip
  
  \noindent{\it Proof for the Addendum to Theorem \ref{T.h}}. The above formulae deliver $4z_nz_{n+1}=(1+(2\gamma_n)^{-1})^2(2\gamma_n)$, i.e. $4z_nz_{n+1}=(1+z_n+z_{n+1})^2(z_n+z_{n+1})^{-1}$. 
  
  For $z$ a point of the elliptic curve  $H$ (with origin $o=\infty_+$), consider the function  $\xi(z)=t(z)+t(z+\delta)$. Observe that $z_n=t((n+1)\delta)$, so the last identity means that the functions $4t(z)t(z+\delta)$ and $(1+\xi(z))^2/\xi(z)$ take the same values at $z=n\delta$ for all $n>0$. Therefore the two functions must coincide, i.e. we have the identity (which could be proved directly) 
  $$
  4t(z)t(z+\delta)={(1+\xi(z))^2\over \xi(z)}.
  $$
  Now, $t$ has (simple) poles at the origin $o$ and at $\delta$, so $t(z+\delta)$ has poles at $-\delta ,o$ and $\xi(z)$ has poles at $-\delta, \delta$ and maybe $o$. But if it had three distinct poles, it would have three zeros and the function on the right side of the identity would have at least $6$ poles, whereas it has degree $\le 4$. We conclude that $\xi(z)$ has poles only at $-\delta,\delta$ and inspection of the identity shows that it must have in fact a double zero at $o$, namely  $\div(\xi)=2(o)-(\delta)-(-\delta)$. \footnote{Of course one could check directly these conclusions.} 
  
 Set now $\eta(z)=\xi(z+\delta)/\xi(z)$.  It has divisor $\div(\eta)= 3(-\delta)+(\delta)-3(o)-(-2\delta)$. 
 
 Also, fix   an integer $k>0$ and set $\pi_k(z)=\eta(z)\eta(z-2\delta)\cdots \eta(z-2k\delta)$. From the above, we readily find that $\div_\infty(\pi_k)= 3(o)+4((o)+(2\delta)+\ldots +((2k-2)\delta))+(2k\delta)$, so $\pi_k$ has degree $4(k+1)$.
  
  It follows from standard height theory that $h(\pi_k(n\delta))\sim 4(k+1)n^2\hat h(\delta)$, (where $h$ is the Weil height  and $\hat h$ is a canonical height associated to a point). 
  
  Now, we have seen that $c_{n}c_{n+1}=\xi((n+1)\delta)$, hence $c_{n+1}=c_{n-1}\eta(n\delta) $, and by iteration it follows that  $c_{n+1}=c_{n-1-2k}\pi_k(n\delta)$. We conclude that, for fixed $k$ and large $n$ we have 
  $$
  h(c_{n+1})+h(c_{n-1-2k})\ge kn^2\hat h(\delta)\gg kn^2,
  $$
proving what we want. Recall  that $c_nc_{n+1}$ has instead height $\ll n^2$.
  
  We finally observe that the recurrences easily yield $h(c_n)\ll n^3$; in the converse direction,  from some quantitative  form of the above height inequality it is probably possible to prove $\max_{m\le n}h(c_m)\gg n^{2+e}$ for some $e>0$. One may ask whether it is possible to take $e=1$ or at least any $e<1$. (Calculations of Merkert would support this expectation.) 
  \end{example} 
  \end{small}

   %%%%%%%%%%%%%%%%%%%%%%%%%%%%%%%%%%%%%%%%%%%%%%%%%%%%%%%%%%%%%%%%%%%%%%%%%%%%%%%%%%%%%%%%%%%%%%%%%%%%%%%%%%%%%%%%%%%%%%%%%%%%%%%%%%%%   THM  ZEROS
  %%%%%%%%%%%%%%%%%%%%%%%%%%%%%%%%%%%%%%%%%%%%%%%%%%%%%%%%%%%%%%%%%%%%%%%%%%%%%%%%%%%%%%%%%%%%%%%%%%%%%%%%%%%%%%%%%%%%%%%%%%
 
\medskip

\subsection{Proof of Theorem \ref{T.zeros}}\label{SS.T.zeros}   We preserve the above notation, letting $D(t)=D_1(t)^2\tilde D(t)$, %, with  squarefree  $D_1$ \footnote{As mentioned in the introduction, this restriction is not essential and is adopted for simplicity.} 
and we let $\rho\in \kappa$ be such that $D(\rho)\neq 0$.  We also denote by $\xi_\pm$ the two points of $\tilde H$ above $t=\rho$. %We stick to non-Pellian $D(t)$, 
The easy Pellian case has been discussed in the Introduction, but the present proof works in that case  as well.

Now,  as in \S \ref{SSS.gj}, $D(t)$ gives us an  extension of $J$ (depending on $D_1$) which we denote by $G$; also,  the modulus $\xi_++\xi_-$ yields an extension of $J$ by $\G_{\rm m}$. We let $\cG$ denote the fiber product over $J$ of these extensions. Hence $\cG$ is an extension of $G$ by $\G_{\rm m}$.  Equivalently, $\cG$ is the extension of $J$ obtained as above, corresponding to the polynomial $\cD(t):=(t-\rho)^2D(t)$.  We shall denote by $\pi:\cG\to G$ the natural map and by $\tt m$, resp. $\cM$, the  moduli corresponding to $G$, resp.  $\cG$.

\medskip

Let us suppose that infinitely many convergents $(p_n,q_n)$ to $\sqrt{D(t)}$ have $q_n$   divisible by $(t-\rho)$, and let us move in a progression (using Theorem \ref{T.dega}) for which $l=a_n$ is fixed. 

As in previous proofs, it may happen that $p_n$ is not coprime to $D_1$. In this case, we divide out by the $\gcd(p_n,D_1)$ (which has only finitely many possibilities) and argue on replacing  $D$ by its corresponding divisor. We may then suppose directly that  $\gcd(p_n,D_1)=1$, so that  $D_1(t(x))\neq 0$ for any for   zero $x\in \widetilde H$ of $\varphi_n:=p_n-q_n\sqrt D$.  %none among the $t(x_i)$ we have that  $t(x)$ is a zero of $D_1$. 

Then we have our usual equation
\begin{equation*}
(\deg p_n)\delta=[x_1]+\ldots +[x_{d-l}]\quad \hbox{in $G$},
\end{equation*}
where  the $x_i$ are points in $\widetilde H\setminus \rm{supp}(\tt m)$, of course depending on $n$, they  are $\neq \infty_\pm$ and no pair of conjugate ones (under the involution) appear.

Now, if $t-\rho$ divides $q_n$  we may set $q_n(t)=(t-\rho)\hat q_n(t)$, and the pair $(p_n,\hat q_n)$ is a convergent to $\sqrt{\cD(t)}$. Also,  $p_n(\rho)\neq 0$, so no $x_i$ is $\xi_\pm$,  and then (since $\varphi_n=p_n-\hat q_n\sqrt{\cD(t)}$) we have the same equation as above, but now in $\cG$:  
\begin{equation*}
(\deg p_n)\delta=[x_1]+\ldots +[x_{d-l}]\quad \hbox{in $\cG$}.
\end{equation*}
Note that $d$ has not been replaced by $d+1$ (and the $x_i$ remain the same); this entails that  the partial quotient now has degree $l+1$, and this represents the `advantage' which we shall exploit. 

\medskip

Now, by Theorem \ref{T.dega}  and Corollary \ref{C.SML}  (both applied to $\cG$)  these equations have to hold for a full arithmetical progression $c+\N\cdot \Pi$ of integers $\deg p_n$, and the Zariski closures in $\cG$, resp. $G$,  of the corresponding multiples of $\delta$ shall contain a coset  $c\delta+\Delta_0(\cM)$ in $\cG$, resp. $c\delta+\Delta_0(\tt m)$ in $G$, and clearly $\pi(\Delta_0(\cM))=\Delta_0(\tt m)$  (indeed, the image is a connected subgroup).  

The above says in particular that a Zariski-dense subset of $c\delta+\Delta_0(\cM)$ is contained in   $W_{d-l}(\cM)$; %, so the variety $\overline{W_{d-l}(\cM)}$; 
recall also  that $d-l<(d+1)-l=\deg\cD-l\le \deg \cD-1$. We shall use this to prove the following crucial 

\begin{lem}\label{L.dim}  
We have $\dim \Delta_0(\cM)=\dim\Delta_0(\tt m)$.
\end{lem}
\begin{proof}[Proof of lemma]  We note that $\pi$ restricts to a surjective homomorphism $\pi:\Delta_0(\cM)\to \Delta_0(\tt m)$, with kernel a subgroup of $\G_{\rm m}$. If the relevant dimensions are different, then $\G_{\rm m}$ (viewed as the kernel of $\pi$ on the whole $\cG$) is contained in $\Delta_0(\cM)$. %, and hence the fiber of $\pi$ restricted to $c\delta +\Delta_0(\cM)$, above any point in $c\delta+\Delta_0(\tt m)$,  is isomorphic to a translate of $\G_{\rm m}$. 

We know that  $W_{d-l}(\cM)$ is a constructible set, i.e. a finite union   $\bigcup_{i\in I}(X_i\setminus Y_i)$, with $Y_i\subset X_i$   closed subvarieties of $\cG$,  and  that  it contains a Zariski-dense subset  $\cZ$ of $c\delta+\Delta_0(\cM)$ (so $c\delta+\Delta_0(\cM)$ is contained in the closure of $W_{d-l}(\cM)$). Hence $\cZ$ %this Zariski-dense subset
  is already contained in the union $\bigcup_{i\in I_1}(X_i\setminus Y_i)$ over the subset $I_1$ of $I$ made up of the $i$ such that $Y_i$ does not contain $c\delta+\Delta_0(\cM)$.
Therefore the closure $c\delta+\Delta_0(\cM)$  is contained in the corresponding finite union   $\bigcup_{i\in I_1}X_i$, whence  a non-empty  open subset $\O$ of $c\delta+\Delta_0(\cM)$ is  contained   too  in the union $\bigcup_{i\in I_1}(X_i\setminus Y_i)$, and hence in   $W_{d-l}(\cM)$. Also, for each point $\theta\in\O$, $\theta+\G_{\rm m}$ is contained in $c\delta+\Delta_0(\cM)$ and then a neighborhood of $\theta$ in $\theta+\G_{\rm m}$ shall be also contained in $\O$  %(This holds because $c\delta+\Delta_0(\cM)$ is irreducible.)  
 and hence in $W_{d-l}(\cM)$.

%Now, by the last displayed equation, all the multiples $(\deg p_n)\delta$ in question are contained in  $W_{d-l}(\cM)$, which is a constructible set, i.e. a finite union of open subsets $X_i\setminus Y_i$, with $Y_i\subset X_i$   closed subvarieties of $\cG$.  Since these multiples are Zariski-dense in $c\delta+\Delta_0(\cM)$, an open subset of 

%We may assume this set of multiples is infinite, for otherwise $\dim \Delta_0(\cM)=0$ and the lemma would be true.

Let then $\theta=s\delta$, $s:=\deg p_n$,  be one of the above multiples of $\delta$ contained in $\O$ (it exists since these multiples are Zariski-dense in $c\delta+\Delta_0(\cM)$).  %, so $s\delta \in X_i\setminus Y_i$ for some $i$. Then $s\delta+\G_{\rm m}$ is not entirely contained in $Y_i$, hence 
%contained in one of the said open sets  of positive dimension;  so since $s\delta+\G_{\rm m}$ is contained in the closure of $W_{d-l}(\cM)$, 
Then  an open neighborhood of $s\delta$ in   $s\delta+\G_{\rm m}$ shall be contained in $W_{d-l}(\cM)$. Take another element    in such open subset of $s\delta+\G_{\rm m}$, represented say by a sum of divisor classes $[y_1]+\ldots +[y_{d-l}]$, $y_i\in H\setminus {\rm supp}(\cM)$. % (where $S$ is the support of $\cM$). 
Then the difference $\sum [x_i]-\sum [y_i]$  is inside $\G_{\rm m}$, which contains  precisely the  divisor classes which are principal and sent to $0$ in $G$, i.e. divisor classes (in the strong sense)  of functions of the shape $a(t)+b(t)\sqrt{D(t)}$, with no zero or pole inside the support of $\cM$. %$S\setminus \{\xi_\pm\}$. 
But if such a function has degree $<d$ it must be in $\C(t)$, by Lemma \ref{L.unique}. We conclude that $\sum (x_i)-\sum (y_i)$  (which is $\neq 0$)  is the  divisor of a nonconstant function in $\C(t)$ and in particular is invariant by the involution. However this is excluded by the above conditions on the $x_i$ (non-ramified points cannot appear and ramified ones can appear at most with multiplicity $1$), which proves finally the lemma.
\end{proof}

The lemma implies in particular that the restriction $\cR:=\pi^{-1}(\Delta_0(\tt m))$ of the extension   $\cG$ of $G$   above $\Delta_0(\tt m)$,  is `almost split', in the sense that it is isogenous to a split one;  indeed, we have clearly an isogeny $\G_{\rm m}\times \Delta_0(\cM)\to \cR$ induced by the inclusion maps. 

At least when $D_1=1$, i.e. $G=J$, it may be proved that this corresponds to the point $\xi_+-\xi_-$ having torsion image in the dual abelian variety $\widehat{\Delta}_0$ (after identifying $J$ and $\widehat J$).\footnote{I  thank  Daniel Bertrand for confirming and clarifying this point and for related  indications.}  We could exploit this fact for our proofs; however this would be somewhat lengthy  because it is not easy to locate references in the literature, and moreover   here we would need the analogue statements for generalized Jacobians. Hence we shall follow another path, and shall only  add some detail for this method in Remark \ref{R.ber} below.

\medskip

By the lemma, the restriction of $\pi$ induces an isogeny $\pi: \Delta_0(\cM)\to \Delta_0(\tt m)$, and let $F$ be the kernel, a finite subgroup of $\G_{\rm m}$. %, hence defined over $\Q$. %Note that if $\kappa$ is a number field of definition for $\cM$ and $\cD(t)$, then $\delta$ is also defined over $\kappa$ and hence the same is true for $\Delta_0(\cM),\Delta_0(\tt m)$ and $\pi$. %Hence $F$ too is defined over $\kappa$. 

Let now $\cR$ be as above, and pick $z\in \cR$. There exists $x\in \Delta_0(\cM)$ with $\pi(x)=\pi(z)$, since $\pi$ is surjective.  Hence  $z-x\in\G_{\rm m}$. This difference depends on the choice of $x$, but another choice yields a translation by an element of $F$. Hence  %the image of $z-x$ in $\G_{\rm m}/F$
 $|F|(z-x)\in\G_{\rm m}$  is well-defined and gives us a homomorphism $\psi: \cR\to \G_{\rm m}$. The kernel is clearly $\Delta_0(\cM)$. 

By Lemma \ref{L.SML} we know that $\Delta_0(\tt m)$ (resp. $\Delta_0(\cM)$) is the Zariski closure of the multiples $\mu\Z\delta$ (resp. $\nu\Z\delta$) for certain positive integers $\mu,\nu>0$; we may suppose that $\mu$ (resp. $\nu$) is the minimal positive integer such that $\mu\delta$ (resp. $\nu\delta$) belongs to $\Delta_0(\tt m)$ (resp. $\Delta_0(\cM)$), and then clearly $\mu$ divides $\nu$, say $\nu=h \mu$. 

Let us consider the homomorphism $\psi$ just introduced, restricted to the group $\mu\Z\delta\subset \cR$. The kernel is $\nu\Z\delta$. %\cap\Delta_0(\cM)$. 
Let $\kappa_0\subset \kappa$ be a field of definition for $H,\delta$ and $\tt m$, so also for $G$, and so $\kappa_1:=\kappa_0(\rho,\xi_\pm)$, which is an extension of $\kappa_0$ of degree $\le 2[\kappa_0(\rho):\kappa_0]$,  is a field of definition also for $\cM$ and $\cG$. We also see that $\psi$ is defined over
at most a quadratic extension $\kappa_2$ (depending possibly on $\rho$) of $\kappa_1$: indeed, the domain $\cR$ and kernel $\Delta_0(\cM)$ are defined over $\kappa_1$, and any variety isomorphic to $\G_{\rm m}$ over some extension, is already isomorphic to $\G_{\rm m}$ over a quadratic extension.\footnote{In fact, this follows since $\G_{\rm m}$ has only $\pm 1$ as automorphism for the algebraic group structure, or else since it has only two points at infinity. %One could also argue directly on looking at the action of $\psi$ on a generic point. 
We also note that one could prove that for this case $\psi$ is  actually defined over $\kappa_1$, even if this is not needed.}   Then $\psi(\mu\delta)$ is defined over $\kappa_2$. However $\psi(\mu\delta)$ is a root of unity of exact order $h$. Hence the $h$-th cyclotomic field is contained in $\kappa_2$.  We conclude that $\varphi(h)\le [\kappa_2:\Q]$, so (for fixed $D(t)$)  $h$ is bounded if the degree of $\rho$ over $\Q$ is bounded.

Let us then suppose that there are infinitely many $\rho$ of bounded degree over $\Q$ and zeros of infinitely many convergents. Then $h$ would be bounded for all of them, and we conclude that (for given $\tt m$)  the relevant progression $\nu\Z$ associated to $\Delta_0(\cM)$ would be the same for infinitely many of these numbers $\rho$. 

Also, for suitable $c$, $(c+\nu m)\delta$ would lie in $W_{d-l}(\cM)$ for every corresponding $\cM=\cM_\rho$ and  large enough $m$ (in terms of $\cM_\rho$); this is because of Corollary \ref{C.SML} which states that the difference  of the relevant progressions is  the same difference related to $\Delta_0(\cM)$. Hence for all large $m$ the divisor $(c+m\nu)\delta$ would be strongly equivalent (relative to the modulus $\cM_\rho$) to a sum of $d-l$ points coprime to $\cM_\rho$, and these points would be uniquely determined independently of $\rho$, because of Lemma \ref{L.unique} applied to $D(t)$.  Select then $K$ of these $\rho$. Then for all large $m$ the above would hold correspondingly to all of these $\rho$, hence by taking the difference of successive elements in the progression we   conclude that $\nu\delta$ is  strongly equivalent to a difference $\sigma:=\sum_{i=1}^{d-l}((x_i)-(y_i))$ relative to each equivalence class corresponding to anyone of the involved $\rho$. But this implies that the strong equivalence holds for the modulus obtained by summing the $\rho$; in other words, $\nu\delta-\sigma$ is the divisor of a nonzero function $a(t)+b(t)u$, where $b(t)$ is divisible by $\prod(t-\rho)$ over these $K$ values of $\rho$ and where $a(t)$ is coprime to this product. But now for $K>\nu+d$ Lemma \ref{L.unique} yields a contradiction, which proves finally the theorem.

\medskip

\begin{small}
\begin{rem}\label{R.ber}  (i) The method yields indeed a sharper result, i.e. the finiteness of the relevant $\rho$ such that the degree of the maximal cyclotomic subfield of  $\kappa_0(\xi)$ is bounded (rather than the degree itself). We note that cyclotomic fields appear, similarly to the easy Pellian case.

\begin{comment}
%(ii) The proof could be concluded also on picking precisely the zeros of the first $\nu$ convergents; however we would have to prove that if $c\delta$ lies in $\overline{W_{d-l}(\cM)}$  then there is a function $a(t)+b(t)\sqrt{D(t)}$, $a,b\in\C[t]$ having divisor $-c\delta+\sigma$, where $\sigma$ is a sum of $d-l$ points (possibly in $S$); this is not entirely automatic, because the $W_m(\cM)$ are not closed subvarieties in general.  
In fact one may show that divisors in the Zclosure of W_m may be represented as sums of m points possibly in S + a principal divisor of a function $a+bD_1u)/c$, where a,b,c are polynomials. (This can be obtained for instance using a limit moving in a curve inside $W_m$.)

One applies this to `small' multiples of $\delta$. Then if $a,b,c$  are coprime, again we have $(t-\rho)$ a zero of $b$ and one may check whether this happens for the small multiples. If they are not coprime then the partial quotient increases in degree; but these points may be excluded since the beginning because $\rho$ does not anymore appear in the divisor and we are dealing with classes in $G$ (not $\cG$). 
\end{comment}

(ii) We illustrate the alternative method alluded to above, supposing that $D(t)$ is squarefree and also, for clarity, that the Jacobian $J$ is simple. In this case the canonical $\Delta_0$ equals $J$ (since we are assuming $D$ non-Pellian).  It is known that the classes of extensions of an abelian variety $A$  by $\G_{\rm m}$ is isomorphic to $Pic^o(A)$ (see \cite{SeAGCF}, Thm. 6, p. 184). Now, the group $Pic^o(A)$ is the underlying group of the dual abelian variety $\hat A$, and since a Jacobian is self-dual it is isomorphic to $J$ in the present case. The extension coming from  $\rho$ is checked to correspond to the point $\xi_+-\xi_-$ in $J$ (see Bertrand's paper \cite{Ber}, \S2.1). This yields an extension $\cG$ which is isogenous to a split  one if and only if $\xi_+-\xi_-$ is a torsion point in $J$. An application of Lemma \ref{L.dim} then says that $\rho$ can be a zero of infinitely many $q_n$ only in this case. On the other hand, for $\dim J>1$ this can happen only finitely many times (by a theorem of Hindry \cite{H} generalising Manin-Mumford's conjecture). In this way we have an improved finiteness result.

This method (which is similar to what already appears in Example \ref{EX.g=1}) could be applied more generally (giving often strong finiteness), but we would have to develop a criterion for isogeny to a split extension above an abelian subvariety, and moreover not merely for Jacobians but for general extensions. Since the result we have proved is sufficient for some applications, and since it uses a completely different method, we have preferred to follow the above path, and we plan to develop the other method in a future paper. 
\end{rem}
\end{small}

  %\medskip

  %%%%%%%%%%%%%%%%%%%%%%%%%%%%%%%%%%%%%%%%%%%%%%%%%%%%%%%%%%%%%%%%%%%%%%%%%%%%%%%%%%%%%%%%%%%%%%%%%%%%%%%%%%%%%%%%%%%%%%%%%%%%%%%%%%%%   THM   IRRED
  %%%%%%%%%%%%%%%%%%%%%%%%%%%%%%%%%%%%%%%%%%%%%%%%%%%%%%%%%%%%%%%%%%%%%%%%%%%%%%%%%%%%%%%%%%%%%%%%%%%%%%%%%%%%%%%%%%%%%%%%%%
 
\subsection{Proof of Theorem \ref{T.irr}} \label{SS.T.irr}   By contradiction, let be given an infinite sequence  $\Sigma$ of positive  integers such that $R_n(t)$ has at least two irreducible factors (with multiplicity) over the number field $\kappa$ (a field of definition for $D$), which are all  distinct for distinct  $n$ varying  in $\Sigma$. 

%We start with a lemma, whose proof is similar to that of Theorem \ref{T.dega2}. \begin{lem} For squarefree $D(t)$, suppose that  for an integer $m<d/2$ the variety $W_m$ contains a translate of an abelian subvariety. Then \end{lem}

We omit the index $n$ and we write as usual $\varphi=p-qu$ and recall the equation $(\deg p)\delta=\sum j(x_i)$, valid in the Jacobian $J$, where the right hand side is a sum of at most $g=d-1$ points $j(x_i)$ where $x_i\in H$ are points distinct from $\infty_\pm$ and such that no conjugate pair $x,x'$ appear. As already noted, by Lemma \ref{L.unique} this also yields that the $x_i$ are uniquely determined by $\deg p$. In particular, we obtain that the divisor $\chi =\chi_n:=\sum (x_i)$ is invariant by Galois action over $\kappa$.

We operate a preliminary step as follows: if some point  among the $x_i$ appears in an infinite subsequence, we go to such a subsequence, and we continue in such a way for the remaining points. So we can eventually write $\chi$ as a sum $\chi=\chi_0+\chi_1$ of two similar effective divisors, where $\chi_0$ is fixed along the whole subsequence and where no point in $\chi_1$ appears infinitely many times. We can also enlarge $\kappa$ to a finite extension so to suppose that $\chi_0$ is defined over $\kappa$. 

We have $R(t)=R_n(t)=c_n\prod (t-t(x_i))$; hence each irreducible factor corresponds to an effective  divisor $\omega\le \chi$,    such that the set $\{t(x_i)\}$, for $x_i$ in the support of $\omega$,  is invariant by Galois action over $\kappa$.  Note that the points in $\chi_0$ shall give factors of degree $1$ taken from a finite set, and the other factors come from divisors $\omega\le \chi_1$.  Hence we can assume that the two irreducible factors in question come from $\chi_1$. 

Also, since no pair of conjugate points appears in $\chi$, we  deduce that the $t$-value individuates uniquely the point, and so  this set  of   $t$-values individuates uniquely the support.   Hence the divisors $\omega$ (and hence $\chi_1$)  are also invariant by Galois action, and the corresponding sum in $J$ is thus in $J(\kappa)$. 
 
But then the former Lang's conjecture, proved by Faltings in \cite{Fa}, \cite{Fa2}, implies that the Zariski-closure of all of  these points consists of  the union of finitely many cosets of  abelian subvarieties of $J$.  
 
 Now we again stick to the divisor  $\chi_1$. The sum of the degrees of the divisors $\omega\le \chi_1$ which arise from the various irreducible factors is $\deg\chi_1\le\deg\chi\le d-1$, so since we are assuming that there are two irreducible factors coming from $\chi_1$,  in particular the smallest such degree $s$ is $<d/2$. On going to an infinite subsequence, we can suppose that $s$ is fixed and that all the rational points corresponding to this minimal degree lie in  a same coset, say $a+A$, of the abelian subvariety $A$ of $J$, and that $a+A$ is their Zariski-closure. Also, these points lie in $W_s$, hence $a+A\subset W_s$.\footnote{Indeed, note that $W_s$ is Zariski-closed. For generalized Jacobians  this would not hold, causing a mild complication.}   Then, if $\xi_1,\xi_2,\zeta$ are  three effective divisors of degree $\le s$ such that $j(\xi_i), j(\zeta)\in  a+A$,  then $ j(\xi_1)+j(\xi_2)-j(\zeta)\in a+A\subset W_s$, so $\xi_1+\xi_2-\zeta$ is linearly equivalent to a sum of $s$ points on $H$, an effective  divisor $\theta$ of degree $s$.  In conclusion, we obtain a linear equivalence $\xi_1+\xi_2\sim \zeta+\theta$, which by Lemma \ref{L.unique} (and since $s<d/2$), yields that this is the divisor of a function in $\C(t)$. 
 
 We can now take $\xi_1=\xi_2=\xi=$ a fixed one among the above divisors $\omega$, and we let $\zeta$ vary  among the effective divisors of degree $s$ such that $j(\zeta)\in a+A$. The fact that  $2\xi-\zeta-\theta$ is the divisor of a function in $\C(t)$ implies that it is invariant by conjugation, and we easily deduce that some point in the support of any of the $\zeta$ belongs to a fixed finite set. But this was excluded by the above opening step, and we have the desired contradiction, concluding the main part of the proof.

 The same argument shows that the degree of this irreducible factor must be $\ge d/2$. 
 
 Finally, if the same irreducible factor appears in $R_m,R_n$, $m<n$, then the ratio $\varphi_n/\varphi_m$ has a nonzero divisor of the shape $h\delta+\omega$, with $h=n-m$ and $\omega$ difference of divisors of bounded degree and support. If we have sufficiently many such equations, with pairwise distinct $h$, we must find the same $\omega$ and subtraction yields that $\delta$ is torsion in $J$, against the assumption that $D$ is non-Pellian. This concludes the proof.

\medskip

\begin{small}
\begin{rem}
Inspection  shows that we can gain some additional precision (e.g. concerning periodic patterns and also proving that $\Phi$ does not depend on $\kappa$ for large $n$), on which we do not comment here. 

Also, dimensional considerations point out that there should exist  cases when some (fixed)  factor of $R_n$ appears infinitely many times, so $R_n$ is not itself  irreducible. (This  happens e.g. for $D(t)=t^8 - 7t^7 + (53/4)t^6 + (3/2)t^5 - (69/4)t^4 + (3/2)t^3 + (53/4)t^2 -7t+ 1$ and  the factor $t-1$.)  For space reasons we postpone any discussion of this to a possible future paper. 
 \end{rem}
 \end{small}

%%%%%%%%%%%%%%%%%%%%%%%%%%%%%%%%%%%%%%%%%%%%%%%%%%%%%%%%%%%%%%%%%%%%%%%%%%%%%%%%%%%%%%%%%%%%%%%%%%%%%%%%%%%%%%%%%%%%%%%%%%%%%%%%%%%%%%%%  BIBLIOGRAPHY %%%%%%%%%%%%%%%%%%%%%%%%%%%%%%%%%%%%%%%%%%%%%%%%%%%%%%%%%%%%%%%%%%%%%%

\vfill

Umberto Zannier

Scuola Normale Superiore

Piazza dei Cavalieri, 7 -  56126 Pisa - ITALY

u.zannier@sns.it

\end{document}